\documentclass[12pt]{article}

\usepackage{amsmath,amsthm,amsfonts,amssymb}
\RequirePackage[colorlinks,citecolor=blue,urlcolor=blue]{hyperref}
\usepackage{color}
\usepackage{url}

\numberwithin{equation}{section}

\newtheorem{theorem}{Theorem}[section]

\newtheorem{proposition}[theorem]{Proposition}
\newtheorem{lemma}[theorem]{Lemma}
\newtheorem{definition}[theorem]{Definition}
\newtheorem{remark}[theorem]{Remark}

\def\cD{\mathcal{D}}

\def\cF{\mathcal{F}}

\def\cH{\mathcal{H}}

\def\cS{\mathcal{S}}

\newcommand{\ud}{\ensuremath{\mathrm{d}}}
\newcommand{\Norm}[1]{\left|\left|  #1   \right|\right|}
\newcommand{\E}{\mathbb{E}}
\newcommand{\e}{\mathrm{e}}

\def\bC{\mathbb{C}}

\def\bP{\mathbb{P}}
\def\bR{\mathbb{R}}
\def\R{\mathbb{R}}

\begin{document}

\title{Parabolic Anderson Model with \\
space-time homogeneous Gaussian noise \\
and rough initial condition}

\author{Raluca M. Balan\footnote{Corresponding author. University of Ottawa, Department of Mathematics and Statistics.
585 King Edward Avenue, Ottawa, ON, K1N 6N5, Canada. {\it E-mail:} \url{rbalan@uottawa.ca}.} \footnote{Research supported by a
grant from the Natural Sciences and Engineering Research Council
of Canada.}
\and
Le Chen\footnote{University of Kansas, Department of Mathematics. 405 Snow Hall, 1460 Jayhawk Blvd. Lawrence, Kansas, 66045-7594, USA. {\it E-mail}: \url{chenle@ku.edu}.}}

\date{June 28, 2016}
\maketitle

\begin{abstract}
\noindent In this article, we study the Parabolic Anderson Model driven by a space-time homogeneous Gaussian noise on $\mathbb{R}_{+} \times \mathbb{R}^d$, whose covariance kernels in space and time are locally integrable non-negative functions, which are non-negative definite (in the sense of distributions). We assume that the initial condition is given by a signed Borel measure on $\mathbb{R}^d$, and the spectral measure of the noise satisfies Dalang's (1999) condition. Under these conditions, we prove that this equation has a unique solution, and we investigate the magnitude of the $p$-th moments of the solution, for any $p \geq 2$. In addition, we show that this solution has a H\"older continuous modification with the same regularity and under the same condition as in the case of the white noise in time, regardless of the temporal covariance function of the noise.
\end{abstract}

\noindent {\em MSC 2010:} Primary 60H15; Secondary 60H07, 37H15.


\vspace{1mm}

\noindent {\em Keywords:} stochastic partial differential equations, rough initial conditions, Parabolic Anderson Model,
Malliavin calculus, Wiener chaos expansion.

\section{Introduction}

In this article, we consider the stochastic heat equation:
\begin{equation}
\left\{\begin{array}{rcl}
\displaystyle \frac{\partial u}{\partial t}(t,x) & = & \displaystyle \frac{1}{2}\Delta u(t,x)+\lambda u(t,x)\dot{W}(t,x), \quad t>0, x \in \bR^d, \:\\[2ex]
\displaystyle u(0,\cdot) & = & u_0(\cdot), \\[1ex]
\end{array}\right.
\label{E:heat} 
\end{equation}
with $\lambda \in \bR$, driven by a zero-mean Gaussian noise $\dot{W}$ defined on a probability space $(\Omega,{\cal F},\bP)$, whose covariance is given informally by:
\begin{align}\label{E:CovW}
 \E\left[\dot{W}(t,x)\dot{W}(s,y)\right] = \gamma(t-s)f(x-y),
\end{align}
for some non-negative and non-negative definite functions $\gamma$ and $f$. The functions $\gamma$ and $f$ are the Fourier transforms (in the sense of distributions) of two tempered measures $\nu$, respectively $\mu$, and hence
the noise is homogeneous in both time and space, i.e. its covariance is invariant under translations. The wave equation with the same type of noise and constant initial conditions has been recently studied in \cite{balan-song16}.

This problem is known in the literature as the {\em Parabolic Anderson Model},
which refers to the fact that the noise enters the equation multiplied by the function $\sigma(u)=\lambda u$.
We are interested in the existence and properties of the random-field solution $u=\{u(t,x); t >0, x \in \bR^d\}$ of equation \eqref{E:heat} interpreted in the Skorohod sense. This means that the solution is defined using a stochastic integral corresponding to the divergence operator from Malliavin calculus.
We refer the reader to Section \ref{S:Back} below for the rigorous definition of the noise and the solution. The novelty of our investigations lies in the fact that we consider initial data given by a signed Borel measure $u_0$ on $\bR^d$, which satisfies the condition:
\begin{align}
\label{E:InitCond}
\int_{\bR^d}e^{-a|x|^2}|u_0|(\ud x)<\infty \quad \mbox{for all} \ a>0,
\end{align}
where $|x|=(x_1+\dots+x_d)^{1/2}$. Here $|u_0|:= u_{0,+}+u_{0,-}$, where $u_0=u_{0,+}-u_{0,-}$  is the Jordan decomposition and
$u_{0,\pm}$ are two non-negative Borel measures with disjoint support.

The Parabolic Anderson Model was originally studied in \cite{CarmonaMolchanov94} in the case when $d=1$ and $\dot{W}$ is replaced by a space-time white noise. In the recent years, there has been a lot of interest in studying the solutions of stochastic partial differential equations (s.p.d.e.'s) driven by a more general Gaussian noise. When the noise is white in time (i.e. the noise behaves in time like a Brownian motion, so that informally, $\gamma=\delta_0$, where $\delta_0$ is the Dirac distribution at $0$), the stochastic integral used in the definition of the solution can be constructed similarly to It\^o's integral, using martingale techniques. In this case, it is known from \cite{Dalang99} that a large class of s.p.d.e.'s  have random-field solutions, under {\em Dalang's condition:}
\begin{align}
\label{E:Dalang}
 \Upsilon(\beta):= (2\pi)^{-d}\int_{\R^d}\frac{\mu(\ud \xi)}{\beta+|\xi|^2}<+\infty\quad
 \text{for some (and hence for all) $\beta>0$,}
\end{align}
where $\mu$ is the spectral measure of the noise in space (defined by \eqref{E:f-Fourier-mu} below). This class includes the heat and wave equations with a Lipschitz non-linear function $\sigma(u)$ multiplying the noise.
These equations have been studied extensively and their solutions possess many interesting properties (see \cite{dalang-mueller09,dalang-sanz09,FK09EJP,SanzSoleSarra02,sanz-sus15} for a sample of relevant references). Most of these properties have been derived for initial conditions given by functions satisfying certain regularity conditions. Recently, some of these properties have been extended to rough initial data (such as Borel measures on $\bR^d$), in the case of the heat equation driven by a space-time white noise (see \cite{ChenDalang13Holder,ChenDalang13Heat}), or even a Gaussian noise which is white in time and colored in space (see \cite{chen-huang16,CK15Color}). The recent preprint \cite{HLN15} carefully analyzes the solution to the Parabolic Anderson Model driven by a Gaussian noise which is white in time and behaves in space either like Dalang's type noise, or like a fractional Brownian motion with $H \in (\frac{1}{4},\frac{1}{2}]$ (when $d=1$). Moreover, in \cite{HLN15} it is assumed that the initial data is given by a function $u_0$ which satisfies \eqref{E:InitCond} with $u_0(\ud x)$ replaced by $u_0(x)\ud x$.

In the present article, we build upon this theory, by studying the Parabolic Anderson Model driven by a Gaussian noise, which is correlated also in time, with temporal covariance kernel given by a locally integrable function $\gamma$. An example which received a lot of attention in the literature is $\gamma(t)=H(2H-1)|t|^{2H-2}$ with $H \in (\frac{1}{2},1)$. In this case, the noise behaves in time like a fractional Brownian motion with index $H$, and
the stochastic integral used for defining the solution has to be constructed using different techniques (usually, Malliavin calculus). The major difficulty is to show that the sequence of Picard iterations converges. This remains an open problem, in the case of equations containing a Lipshitz non-linear function $\sigma(u)$ multiplying the noise. However, as observed in \cite{hu-nualart09}, this problem has a surprisingly simple solution when $\sigma(u)=\lambda u$. In this case, the solution has an explicit series representation (given by its Wiener chaos expansion), and the necessary and sufficient condition for the existence (and uniqueness) of the solution is that this series converges in $L^2(\Omega)$. This method yields immediately an upper bound for the $p$-the moment of the solution, using the equivalence of the $L^p(\Omega)$-norms on the same Wiener chaos space. When the initial condition is given by a bounded function, this technique was used to investigate the properties of the solutions (see for instance \cite{balan-conus16, HHNT15, C3H16}). In \cite{C3H16}, the heat operator was replaced by general fractional operators in both time and space variables.

The goal of this article is to use the same method based on Wiener chaos expansion to prove the existence of the solution of equation \eqref{E:heat}, with initial data given by a signed measure $u_0$ satisfying \eqref{E:InitCond}. In particular, Dirac delta initial data was used in the theory of
Borodin, Corwin and their coauthors for equations driven by space-time white noise in spatial dimension $d=1$ (see e.g. \cite{borodin-corwin14}).
Since the initial data plays an important role in the form of the kernels $f_n(\cdot,t,x)$ appearing in the series representation of the solution $u(t,x)$
(see \eqref{E:def-fn} below),
new ideas are required to show that this series converges in $L^2(\Omega)$. This leads to calculations that deviate significantly from the case of bounded initial conditions. The starting point of these calculations is an elementary result borrowed from \cite{ChenDalang13Heat} (see Lemma \ref{L:GG} below), which is specific to the heat equation. In fact, $f_n(\cdot,t,x)$ depends on $u_0$ through the solution $J_0$ of the homogeneous heat equation with initial data $u_0$, defined by:
\begin{align}
\label{E:J0}
J_0(t,x)=\int_{\bR^d}G(t,x-y)u_0(\ud y),
\end{align}
where $G(t,x)$ the fundamental solution of the heat equation in $\bR^d$:
$$G(t,x)=\frac{1}{(2\pi t)^{d/2}}\exp\left(-\frac{|x|^2}{2t}\right), \quad t>0,\: x \in \bR^d.$$ Therefore, \eqref{E:InitCond} is the weakest condition one has to impose on $u_0$ to ensure that the series converges. To see this, it suffices to note that the first term of the series is $J_0(t,x)$, and $|J_0(t,x)| \leq J_{+}(t,x)$, where
\begin{align}
\label{E:J+}
J_{+}(t,x)=\int_{\bR^d}G(t,x-y)|u_0|(\ud y).
\end{align}
A simple argument shows that condition \eqref{E:InitCond} is equivalent to
$$J_{+}(t,x) < +\infty\qquad\text{for all $t>0$ and $x\in\R^d$}.$$

After establishing the existence of the solution, we proceed to a careful analysis of the order of magnitude of the $p$-th moments of the solution. This investigation reveals that we have to distinguish between
two different scenarios. When $\gamma$ is integrable on $\bR$, under a slightly stronger requirement on $u_0$ (given by \eqref{E:expo-u0} below), we show that $\E|u(t,x)|^p \leq c_1\exp(c_2 t)$ for $t$ large, uniformly in $x \in \bR^d$, regardless of the spatial covariance kernel $f$. In this case, the smoothness of the noise in time overcomes both the roughness of the noise in space and the roughness of the initial data, leading to the same behaviour of the solution as in the case of equations with space-time white noise and bounded initial condition.
On the other hand, if $f$ is the Riesz kernel of order $\alpha$, $\E|u(t,x)|^p \leq c_1 J_{+}^p(t,x) \exp(c_2 \Gamma_t^{2/(2-\alpha)}t)$, where
\begin{align}
\label{E:Gammat}
 \Gamma_t:=\int_{-t}^{t}\gamma(s)\ud s=2\int_0^t \gamma(s)\ud s.
\end{align}In this case, the behaviour of the solution retains all the characteristics of the noise, combined with the roughness of the initial data.

Proving that the solution is $L^p(\Omega)$-continuous relies on some technical arguments, which we placed in Appendix B to preserve the natural reading flow. Finally, in the last part of the article, we prove that the solution to equation \eqref{E:heat} has a H\"older continuous modification, with the same orders of regularity and under the same condition on the spectral measure $\mu$, as in the case of equations with white noise in time. This shows that neither the correlation of the noise in time, nor the initial data affects the sample path regularity of the solution. A similar fact was observed in \cite{balan-song16} in the case of the wave equation with constant initial conditions. The proof of this result follows from Kolmogorov's continuity criterion, by refining the bounds obtained in Appendix B for the $p$-th moments of the increments of the solution.


The main results of this article are summarized by the following three theorems.
We let $\Norm{\cdot}_p$ be the $L^p(\Omega)$-norm. The rigorous meaning of solution is given by Definition \ref{D:Sol} below.

\begin{theorem}
\label{T:ExUni}
Assume that Dalang's condition \eqref{E:Dalang} holds.

\noindent
(a) Then, for any Borel measure $u_0$ on $\R^d$ that satisfies \eqref{E:InitCond},
equation \eqref{E:heat} has a unique random field solution
$\left\{u(t,x): t>0,x \in\R^d\right\}$.
For any $p \geq 2$,
\begin{align}
\label{E:u-p}
 \Norm{u(t,x)}_p\le J_+(t,x) \: \widetilde{H}\left(t;2\lambda^2(p-1)\Gamma_t \right) ,
\end{align}
where $\Gamma_t$ is defined in \eqref{E:Gammat} and $\widetilde{H}(t;\gamma)$ is defined in \eqref{E:H} below.
In particular, for any $a>1$,
\begin{equation}
\label{E:sup-Ka}\sup_{(t,x) \in K_a}E\left(|u(t,x)|^p\right)<\infty,
\end{equation}
where $K_a=[1/a,a] \times [-a,a]^d$.
Moreover, $u$ is $L^p(\Omega)$-continuous on $(0,\infty) \times \bR^d$
for all $p \geq 2$.

\noindent
(b) If $\gamma\in L^1(\R)$, i.e., $\Gamma_\infty:=\lim_{t\rightarrow\infty}\Gamma_t <\infty$, and the initial measure $u_0$ satisfies
\begin{equation}
\label{E:expo-u0}
\int_{\bR^d} e^{-\beta|x|}|u_0|(\ud x)<\infty \quad \mbox{for all} \ \beta>0,
\end{equation}
then for all $p\ge 2$,
 \[
 \sup_{x\in\R^d} \limsup_{t\rightarrow\infty}\frac{1}{t}\log\Norm{u(t,x)}_p\le \inf\left\{\beta>0 \::\: \Upsilon(2\beta)<\left[4\lambda^2 (p-1)\Gamma_\infty  \right]^{-1}\right\}.
 \]

\noindent
(c) Assume that $\gamma\in L^1(\R)$ and the initial measure $u_0$ satisfies \eqref{E:expo-u0}. If
 \begin{equation}
 \label{E:Upsilon0}
  \Upsilon(0):=\lim_{\beta\rightarrow 0} \Upsilon(\beta)<\infty
 \end{equation}
(which happens only when $d\ge 3$), then there exists some critical value $\lambda_c>0$ such that
 when $|\lambda|<\lambda_c$,
 \[
 \sup_{x\in\R^d} \limsup_{t\rightarrow\infty}\frac{1}{t}\log\Norm{u(t,x)}_p =0.
 \]
\end{theorem}

\begin{theorem}\label{T:Riesz}
Suppose that $\mu(\ud\xi)=|\xi|^{-(d-\alpha)}\ud\xi$ for some $0<\alpha<d\wedge 2$
and the initial measure $u_0$ satisfies \eqref{E:InitCond}.
Then equation \eqref{E:heat} has a unique solution $\left\{u(t,x);t>0, x\in \bR^d\right\}$
which satisfies the following moment bound:
\begin{align}\label{E:Riez-Mom}
\E\left[ |u(t,x)|^p\right] \leq C^p J_+^p(t,x) \exp\left(C p^{(4-\alpha)/(2-\alpha)} |\lambda|^{4/(2-\alpha)} \Gamma_t^{2/(2-\alpha)}t\right),
\end{align}
for all $p\ge 2$, $t>0$ and $x\in\R^d$, where $C>0$ is some universal constant.
\end{theorem}

\begin{theorem}
\label{T:Holder}
Let $u$ be the solution of equation \eqref{E:heat} starting from an initial measure $u_0$ that satisfies \eqref{E:InitCond}. Suppose that:
\begin{equation}
\label{C:beta-cond}
\int_{\bR^d}\left(\frac{1}{1+|\xi|^2}\right)^{\beta}\mu(\ud \xi)<\infty \quad \mbox{for some} \ \beta \in (0,1).
\end{equation}
Then for any $p \geq 2$ and $a>1$ there exists a constant $C>0$
depending on $p,a,\lambda$ and $\beta$ such that for any $(t,x),(t',x')\in K_a:=[1/a,a] \times [-a,a]^d$,
$$\|u(t,x)-u(t',x')\|_p \leq C \left(|t-t'|^{\frac{1-\beta}{2}}+|x-x'|^{1-\beta}\right).$$
Consequently, for any $a>1$, the process $\left\{u(t,x); (t,x) \in K_a\right\}$ has a modification which is a.s.
$\theta_1$-H\"older continuous in time and a.s. $\theta_2$-H\"older continuous in space,
for any $\theta_1 \in (0, (1-\beta)/2)$ and $\theta_2 \in (0,1-\beta)$.
\end{theorem}

\begin{remark}\normalfont
(a) The proof of Theorem \ref{T:ExUni}.(a) shows that, for any $p \geq 2$ and $T>0$,
\begin{equation}
\label{E:p-mom}
\sup_{(t,x)\in [0,T] \times \bR^d}\|u(t,x)\|_p \leq C_{\lambda,p,T} \sup_{(t,x) \in [0,T] \times \bR^d}J_+(t,x)
\:\le +\infty,
\end{equation}
where $C_{\lambda,p,T}>0$ is a constant which depends on $\lambda$, $p$ and $T$.
If $u_0(\ud x)=u_0(x)\ud x$ and $u_0$ is bounded,
then $\sup_{(t,x) \in [0,T] \times \bR^d}J_+(t,x)<\infty$.
But there are many examples of measures $u_0$, such as $u_0=\delta_0$ and $u_0(\ud x)=|x|^2 \ud x$, for which
$\sup_{(t,x) \in [0,T] \times \bR^d}J_+(t,x)=\infty$.

(b) When $u_0(dx)=a dx$ for some $a>0$ and $\gamma_t=H(2H-1)|t|^{2H-2}$ for $H \in (\frac{1}{2},1)$, the upper bound given by Theorem \ref{T:Riesz} coincides with the one of Proposition 8.1.(b) of \cite{balan-conus16}.

(c) In the case of the white noise in time, the H\"older regularity of the solution of the heat equation with initial condition given by $u_0(\ud x)=u_0(x)\ud x$ (with $u_0$ a bounded and H\"older continuous function) was obtained in \cite{SanzSoleSarra02}
under the same condition \eqref{C:beta-cond} and with the same exponents as in Theorem \ref{T:Holder}. This result has been recently extended in  \cite{chen-huang16} to the case of initial data given by a signed measure $u_0$ satisfying \eqref{E:InitCond}. In \cite{ChenDalang13Holder}, it was shown that
the solution of the heat equation with space-time white noise and initial data satisfying \eqref{E:InitCond} has a modification which is $\theta_1$-H\"older continuous in time and $\theta_2$-H\"older continuous in space, for any $\theta_1 \in (0,\frac{1}{4})$ and $\theta_2 \in (0,\frac{1}{2})$. This is consistent with the conclusion of Theorem \ref{T:Holder}, since for the space-time white noise, $d=1$, $\mu(d\xi)=\ud \xi$, and condition \eqref{C:beta-cond} holds for $\beta=1/2+\varepsilon$ with $\varepsilon>0$ arbitrary.

(d) Finding a nontrivial lower bound for the second moment of the solution to equation \eqref{E:heat} when the initial condition is the Dirac delta
measure is an extremely challenging problem. We postpone this for future work. When the noise is white in time, a nontrivial lower bound has been recently obtained in \cite{CK15Color}.
\end{remark}


We conclude the introduction with few words about the organization of the article and the notation. In Section \ref{S:Back}, we introduce the background material necessary for the rigorous formulation of the problem.
Section \ref{S:Exist} is dedicated to the the proof of Theorem \ref{T:ExUni}, while Theorems \ref{T:Riesz} and \ref{T:Holder} are proved in Sections \ref{S:Riesz} and \ref{S:Holder}, respectively. The appendix contains the proofs of some technical results.

Throughout this article, we denote by $\cD(\R^d)$ the set of $C^{\infty}$-functions with compact support in $\bR^d$, and $\cS(\R^d)$ the set of Schwartz test functions on $\bR^d$, i.e. $C^{\infty}$-functions with rapid decrease at infinity along with all partial derivatives. We let $\cS_{\bC}(\R^d)$ be the set of $\bC$-valued Schwartz test functions on $\bR^d$, and $\cS_{\bC}'(\R^d)$ be its dual space. We denote by
$L_{\bC}^1(\bR^d)$ the space of $\bC$-valued integrable functions on $\bR^d$.
The Fourier transform of a function $\varphi \in L_{\bC}^1(\bR^d)$ is defined by:
$$\cF \varphi(\xi)=\int_{\bR^d}e^{-i \xi \cdot x}\varphi(x)\ud x, \quad \xi \in \bR^d.$$
We say that a measure $\mu$ on $\bR^d$ is tempered if
$$\int_{\bR^d}\left(\frac{1}{1+|\xi|^2} \right)^k \mu (\ud \xi)<\infty \quad \mbox{for some} \ k>0.$$

\section{Background}
\label{S:Back}

In this section, we introduce the definitions of the noise and solution, review some basic facts of Malliavin calculus, and give some preliminary results related to the existence of the solution, with emphasis on the Wiener chaos expansion of the solution.

We begin by recalling the definition of the noise from \cite{balan-song16}.
We assume that $W=\{W(\varphi):\:\varphi \in \cD(\bR \times \bR^d)\}$ is a zero-mean Gaussian process, defined on a probability space $(\Omega,\cF,\bP)$, with covariance
\[
\E[W(\varphi_1)W(\varphi_2)]=\int_{\bR^2 \times \bR^{2d}}\gamma(t-s)f(x-y)\varphi_1(t,x)\varphi_2(s,y)\ud x \ud y \ud t \ud s=:J(\varphi_1,\varphi_2),
\]
where $\gamma:\bR \to [0,\infty]$ and $f:\bR^d \to [0,\infty]$ are continuous, symmetric, locally integrable functions, such that
$$\gamma(t) <\infty \quad \mbox{if and only if} \quad t\not=0,$$
$$f(x) <\infty \quad \mbox{if and only if} \quad x\not=0.$$
We denote by $\cH$  the completion of $\cD(\bR \times \bR^d)$ with respect to $\langle \cdot,\cdot\rangle_{\cH}$ defined by
$$\langle \varphi_1,\varphi_2\rangle_{\cH}=J(\varphi_1,\varphi_2).$$
We are mostly interested in variables $W(\varphi)$ with $\varphi \in \cD(\bR_{+}\times \bR^d)$.

We assume that the functions $\gamma$ and $f$ are \emph{non-negative definite} (in the sense of distributions), i.e. for any $\phi \in \cS(\bR)$ and
$\varphi \in \cS(\bR^d)$,
$$\int_{\bR}(\phi*\widetilde{\phi})(t)\gamma(t)\ud t \geq 0 \quad \mbox{and} \quad \int_{\bR^d}(\varphi*\widetilde{\varphi})(x)f(x)\ud x \geq 0.$$

By the Bochner-Schwartz Theorem, there exists a tempered measure $\nu$ on $\bR$ such that $\gamma$ is the Fourier transform of $\nu$ in $\cS_{\bC}'(\bR)$, i.e.
$$\int_{\bR^d}\phi(t)\gamma(t)\ud t=\frac{1}{2\pi}\int_{\bR}\cF \phi(\tau)\nu(\ud \tau) \quad \mbox{for all} \quad \phi \in \cS_{\bC}(\bR).$$
Similarly, there exists a tempered measure $\mu$ on $\bR^d$ such that $f$ is the Fourier transform of $\mu$ in $\cS_{\bC}'(\bR^d)$, i.e.
\begin{equation}
\label{E:f-Fourier-mu}
\int_{\bR^d}\varphi(x)f(x)\ud x=\frac{1}{(2\pi)^d}\int_{\bR^d}\cF \varphi(\xi)\mu(\ud\xi) \quad \mbox{for all} \quad \varphi \in \cS_{\bC}(\bR^d).
\end{equation}

It follows that for any functions $\phi_1,\phi_2 \in \cS_{\bC}(\bR)$ and  $\varphi_1,\varphi_2 \in \cS_{\bC}(\bR^d)$
\begin{equation}
\label{E:Parseval-t}\int_{\bR} \int_{\bR}\gamma(t-s)\phi_1(t)\overline{\phi_2(s)}\ud t \ud s=\frac{1}{2\pi}\int_{\bR}\cF \phi_1(\tau)\overline{\cF \phi_2(\tau)}\nu(\ud \tau)
\end{equation}
and
\begin{equation}
\label{E:Parseval-x}
\int_{\bR^d} \int_{\bR^d}f(x-y)\varphi_1(x)\overline{\varphi_2(y)}\ud x\ud y=\frac{1}{2\pi}\int_{\bR^d}\cF \varphi_1(\xi)\overline{\cF \varphi_2(\xi)}\mu(\ud\xi).
\end{equation}

The next result shows that the functional $J$ is non-negative definite.

\begin{lemma}[Lemma 2.1 of \cite{balan-song16}]
For any $\varphi_1, \varphi_2 \in \cD(\bR \times \bR^d)$, we have:
\begin{equation}
\label{alt-expres-J}
J(\varphi_1,\varphi_2)=\frac{1}{(2\pi)^{d+1}}\int_{\bR^{d+1}}
\cF \varphi_1(\tau,\xi)\overline{\cF \varphi_2(\tau,\xi)}\nu(\ud \tau)\mu(\ud\xi),
\end{equation}
where $\cF$ denotes the Fourier transform in both variables $t$ and $x$. In particular, $J$ is non-negative definite.
\end{lemma}

\bigskip

At this point, we need to introduce some basic facts from Malliavin calculus, which are necessary for defining the solution to equation \eqref{E:heat}.
We refer the reader to \cite{nualart06} for more details.
It is known that every square-integrable random variable $F$ which is measurable with respect to $W$,
has the Wiener chaos expansion:
$$F=\E(F)+\sum_{n \geq 1}F_n \quad \mbox{with} \quad F_n \in \cH_n,$$
where $\cH_n$ is the $n$-th Wiener chaos space associated to $W$. Moreover, each $F_n$ can be represented as $F_n=I_n(f_n)$ for some $f_n \in \cH^{\otimes n}$, where $\cH^{\otimes n}$ is the $n$-th tensor product of $\cH$ and
$I_n:\cH^{\otimes n} \to \cH_n$ is the multiple Wiener integral with respect to $W$.
By the orthogonality of the Wiener chaos spaces and an isometry-type property of $I_n$, we obtain that
$$\E(|F|^2)=(\E F)^2+\sum_{n \geq 1}\E(|I_n(f_n)|^2)=(\E F)^2+\sum_{n \geq 1}n! \|\widetilde{f}_n\|_{\cH^{\otimes n}}^{2},$$
where $\widetilde{f}_{n}$ is the symmetrization of $f_n$ in all $n$ variables:
$$\widetilde{f}_n(t_1,x_1,\ldots,t_n,x_n)=\frac{1}{n!}\sum_{\rho \in S_n}f_n(t_{\rho(1)},x_{\rho(1)},\ldots,t_{\rho(n)},x_{\rho(n)}).$$
Here $S_n$ is the set of all permutations of $\{1, \ldots,n\}$. We note that
the space $\cH^{\otimes n}$ may contain distributions in $\cS'(\bR^{n(d+1)})$.

We denote by $\delta: {\rm Dom}(\delta) \subset L^2(\Omega;\cH) \to L^2(\Omega)$ the divergence operator with respect to $W$, defined as the adjoint of the Malliavin derivative  $D$ with respect to $W$.
If $u \in \mbox{Dom} \ \delta$, we use the notation
$$\delta(u)=\int_0^{\infty} \int_{\bR^d}u(t,x) W(\delta t, \delta x),$$
and we say that
$\delta(u)$ is the {\em Skorohod integral} of $u$ with respect to $W$. In particular, $E[\delta(u)]=0$.

\bigskip
We are now ready to give the definition of the solution to equation \eqref{E:heat}.

\begin{definition}
\label{D:Sol}
{\rm We say that a process $u=\{u(t,x);t \geq 0, x \in \bR^d\}$ is a {\bf (mild) solution} of equation \eqref{E:heat} if for any $t>0$ and $x \in \bR^d$, $u(t,x)$ is $\cF_t$-measurable, $E|u(t,x)|^2<\infty$ and the following integral equation holds:
\begin{equation}
\label{def-sol}
u(t,x) = J_0(t,x) + \int_{0}^{t}\int_{\bR^d}G(t-s,x-y)u(s,y)W(\delta s,\delta y),
\end{equation}
i.e. $v^{(t,x)} \in {\rm Dom} \ \delta$ and $u(t,x)=J_0(t,x)+\delta(v^{(t,x)})$ for all $(t,x) \in \bR_{+} \times \bR^d$, where
\begin{equation}
\label{def-v}
v^{(t,x)}(s,y)=1_{[0,t]}(s)G(t-s,x-y)u(s,y), \quad s \geq 0,\: y \in \bR^d.
\end{equation}
}
\end{definition}

We now state (without proof) a well-known criterion for the existence and uniqueness of this solution, expressed as the convergence in $L^2(\Omega)$ of a series of multiple integrals. This result is essentially due to \cite{hu-nualart09} (for a slightly different noise than here). In its present form, it is similar to Theorem 2.9 of \cite{balan-song16} (for the wave equation). We define the kernel function $f_n(\cdot,t,x)$ by:
\begin{align}
\nonumber
f_n(t_1,x_1,\ldots,t_n,x_n,t,x)=&
\quad \lambda^n G(t-t_n,x-x_n)\ldots G(t_2-t_1,x_2-x_1)\\
\label{E:def-fn}
&\times J_0(t_1,x_1)1_{\{0<t_1<\ldots<t_n<t\}}.
\end{align}

\begin{theorem}
\label{exist-sol-th}
Suppose that $f_n(\cdot,t,x)\in \cH^{\otimes n}$ for any $t>0,x \in \bR^d$ and $n \geq 1$. Then equation \eqref{E:heat} has a solution if and only if for any $t>0$ and $x \in \bR^d$,
$$\mbox{the series $\sum_{n \geq 0}I_n(f_n(\cdot,t,x))$ converges in $L^2(\Omega)$}.$$ In this case, the solution is unique and is given by:
$$u(t,x)=\sum_{n \geq 0}J_n(t,x), \quad with \quad J_n(t,x)=I_n(f_n(\cdot,t,x)).$$
\end{theorem}

To show that the kernel $f_n(\cdot,t,x)$ is in $\cH^{\otimes n}$ we need an alternative expression of this kernel, which is obtained as follows.
Suppose that $0<t_1<\ldots<t_n<t$. Using the definition of $J_0(t_1,x_1)$, we see that
$$f_n(t_1,x_1, \ldots,t_n,x_n,t,x)=\lambda^n \int_{\bR^d} G(t-t_n,x-x_n)\ldots
G(t_2-t_1,x_2-x_1)G(t_1,x_1-x_0)u_0(\ud x_0).$$

The key idea (and the starting point of our developments) is to express the product $G(t_2-t_1,x_2-x_1)G(t_1,x_1-x_0)$ above using the following result, whose proof is based on the specific form of the heat kernel $G$.

\begin{lemma}[Lemma 5.4 of \cite{ChenDalang13Heat}]
\label{L:GG}
For $t,\: s>0$ and $x,\: y\in\bR^d$,
\[
G(t,x)G(s,y)=G\left(\frac{ts}{t+s},\frac{sx+ty}{t+s}\right)G(t+s,x-y).
\]
\end{lemma}

Consequently, we obtain that
\begin{align*}
f_n(t_1,x_1, \ldots,t_n,x_n,t,x)=\lambda^n &\int_{\bR^d}  G(t-t_n,x-x_n)\ldots G(t_3-t_2,x_3-x_2)G(t_2,x_2-x_0)\\
& \times G\left(\left(1-\frac{t_1}{t_{2}}\right)t_1,
\left(1-\frac{t_1}{t_{2}}\right)x_0+\frac{t_1}{t_{2}}x_{2}-x_1\right)u_0(\ud x_0).
\end{align*}
We now express the product $G(t_3-t_2,x_3-x_2)G(t_2,x_2-x_0)$ using Lemma \ref{L:GG}, and we continue in this manner. After $n$ steps, letting $t_{n+1}=t$, we obtain that:
\begin{align}
\begin{aligned}
f_n(t_1,x_1, \ldots,t_n,x_n,t,x)&=
\lambda^n \int_{\bR^d} u_0(\ud x_0) \: G(t,x-x_0) \\
&\hspace{-1em}\times \prod_{j=1}^{n}G\left(\left(1-\frac{t_j}{t_{j+1}}\right)t_j,
\left(1-\frac{t_j}{t_{j+1}}\right)x_0+\frac{t_j}{t_{j+1}}x_{j+1}-x_j\right).
\label{E:alt-fn}
\end{aligned}
\end{align}

Using the fact that $\int_{\bR^d}G(t,x-y)\ud y=1$ for any $x \in \bR^d$, we see that
\[
\left|\int_{\bR^{nd}}f_n(t_1,x_1, \ldots,t_n,x_n,t,x) \ud x_1 \ldots \ud x_n\right|
\le \lambda^n \int_{\bR^d}G(t,x-x_0)|u_0|(\ud x_0)=\lambda^n J_+(t,x)<\infty.
\]
This shows that the function $f_n(t_1,\cdot,\ldots,t_n,\cdot,t,x)$ is in $L^1(\bR^{nd})$. The next result gives the Fourier transform of this function. For this, we need to recall that:
\begin{equation}
\label{E:Fourier-G}
\cF G(t,\cdot)(\xi)=\exp\left(-\frac{t|\xi|^2}{2}\right) \quad \mbox{for all} \ t>0,\: \xi \in \bR^d.
\end{equation}

\begin{lemma}
\label{L:Fourier-fn}
For any $0<t_1<\ldots<t_n<t=t_{n+1}$ and for any $\xi_1, \ldots,\xi_n \in \bR^d$, we have
\begin{align*}
\cF f_n(t_1,\cdot,\ldots,t_n,\cdot,t,x)&(\xi_1,\ldots,\xi_n)\\
=&\lambda^n \prod_{k=1}^{n} \exp\left\{-\frac{1}{2}
\frac{t_{k+1}-t_k}{t_k t_{k+1}} \left|\sum_{j=1}^k t_j \xi_j\right|^2 \right\}\:
\exp\left\{- \frac{i}{t}\left(\sum_{j=1}^{n}t_j \xi_j\right)\cdot x \right\} \\
 & \times \int_{\bR^d} \exp\left\{-i \left[ \sum_{j=1}^{n}\left(1-\frac{t_j}{t}\right)\xi_j \right]\cdot x_0 \right\}G(t,x-x_0)u_0(\ud x_0).
\end{align*}
\end{lemma}

\begin{proof} By definition,
\[
\cF f_n(t_1,\cdot,\ldots,t_n,\cdot,t,x)(\xi_1,\ldots,\xi_n)=\int_{\bR^{nd}}e^{-i (\xi_1 \cdot x_1+\ldots+\xi_n \cdot x_n)} f_n(t_1,x_1,\ldots,t_n,x_n,t,x)\ud x_1 \ldots \ud x_n.
\]
We use the alternative definition \eqref{E:alt-fn} of the kernel $f_n(\cdot,t,x)$ and Fubini's theorem. We calculate first the $\ud x_1$ integral:
\begin{align*}
\int_{\bR^d}e^{-i \xi_1 \cdot x_1} &G\left(\left(1-\frac{t_1}{t_{2}}\right)t_1,
\left(1-\frac{t_1}{t_{2}}\right)x_0+\frac{t_1}{t_{2}}x_{2}-x_1\right)\ud x_1\\
 =& \exp\left\{-i \xi_1 \cdot \left[ \left(1-\frac{t_1}{t_2}x_0+\frac{t_1}{t_2}x_2 \right)\right] \right\}
\\
& \times
 \overline{\cF G\left(\left(1-\frac{t_1}{t_2}\right)t_1,\cdot\right)(\xi_1)}\:,
\end{align*}
where we used the fact that for any $t>0$, $x \in \bR^d$ and $\xi \in \bR^d$,
\begin{equation}
\label{E:Fourier-G-x}
\cF G(t,x-\cdot) (\xi) =\int_{\bR^d}e^{-i \xi \cdot y}G(t,x-y)\ud y= e^{-i \xi \cdot x}\overline{\cF G(t,\cdot)(\xi)}.
\end{equation}
We calculate next the $\ud x_2$ integral, using again \eqref{E:Fourier-G-x}:
\begin{align*}
\int_{\bR^d}\exp\left\{-i \left(\xi_2+ \frac{t_1}{t_2}\xi_1\right)\cdot x_2 \right\} &G\left(\left(1-\frac{t_2}{t_{3}}\right)t_2,
\left(1-\frac{t_2}{t_{3}}\right)x_0+\frac{t_2}{t_{3}}x_{3}-x_2\right)\ud x_2 \\
=&\exp \left\{-i\left(\xi_2+\frac{t_1}{t_2}\xi_1\right)\cdot \left[\left(1-\frac{t_2}{t_3}\right)x_0+\frac{t_2}{t_3}x_3 \right] \right\}
\\
&\times \overline{\cF G\left(\left(1-\frac{t_2}{t_3}\right) t_2,\cdot\right)\left(\xi_2+\frac{t_1}{t_2}\xi_1 \right)}.
\end{align*}
We continue in this manner. At the last step, we obtain the following $\ud x_n$ integral:
\begin{align*}
\int_{\bR^d}
\exp\left\{-i \left(\xi_n+ \frac{\sum_{j=1}^{n-1}t_j \xi_j}{t_n}\right)\cdot x_n \right\}
& G\left(\left(1-\frac{t_n}{t}\right)t_n,
\left(1-\frac{t_n}{t}\right)x_0+\frac{t_n}{t}x-x_n\right)\ud x_n \\
=& \exp \left\{-i \left(\xi_n+ \frac{\sum_{j=1}^{n-1}t_j \xi_j}{t_n} \right) \cdot \left[ \left(1-\frac{t_n}{t} \right)x_0+\frac{t_n}{t}x\right]\right\}
\\
&\times \overline{\cF G \left(\left(1-\frac{t_n}{t}\right) t_n,\cdot\right)\left(\xi_n+\frac{\sum_{j=1}^{n-1}t_j \xi_j}{t_n} \right)}.
\end{align*}
Putting together all these calculations, it follows that $\cF f_n(t_1,\cdot,\ldots,t_n,\cdot,t,x)(\xi_1,\ldots,\xi_n)$ is equal to
\begin{multline*}
\prod_{k=1}^{n} \overline{\cF G \left(\left(1-\frac{t_k}{t_{k+1}} \right)t_k, \cdot\right)\left(\frac{\sum_{j=1}^{k}t_j \xi_j}{t_k}\right)}
\: \exp \left\{
-i\frac{\sum_{j=1}^{n}t_j \xi_j}{t} \cdot x \right\}    \int_{\bR^d} u_0(\ud x_0)\: G(t,x-x_0) \\
 \times
  \exp \left\{
-i \left[\sum_{j=1}^{n-1}\left(1-\frac{t_j}{t_n}\right)\xi_j+\left(1-\frac{t_n}{t}\right)
\left(\xi_n+
\frac{\sum_{j=1}^{n-1}t_j\xi_j}{t_n} \right) \right]\cdot x_0\right\} .
\end{multline*}
We note that
\[
\sum_{j=1}^{n-1}\left(1-\frac{t_j}{t_n}\right)\xi_j+\left(1-\frac{t_n}{t}\right)
\left(\xi_n+
\frac{\sum_{j=1}^{n-1}t_j\xi_j}{t_n} \right)=\sum_{j=1}^{n}
\left(1-\frac{t_j}{t} \right)\xi_j.
\]
Using \eqref{E:Fourier-G}, we have
\[
\cF G \left(\left(1-\frac{t_k}{t_{k+1}} \right)t_k, \cdot\right)\left(\frac{\sum_{j=1}^{k}t_j \xi_j}{t_k}\right)=\exp \left\{-\frac{1}{2}
\left(1-\frac{t_k}{t_{k+1}}\right) t_k \left|\frac{\sum_{j=1}^{k}t_j \xi_j}{t_k}\right|^2  \right\}.
\]
The conclusion follows.
\end{proof}

To apply Theorem \ref{exist-sol-th}, we first need to check that each kernel $f_n(\cdot,t,x)$ lives in the $n$-th Wiener chaos space $\cH_n$ (and hence, its multiple integral with respect to $W$ is well-defined). The following result shows that Dalang's condition \eqref{E:Dalang} on the spatial spectral measure of the noise is sufficient for achieving this, regardless of the temporal covariance function $\gamma$.

\begin{theorem}
\label{T:fn-in-H}
If $\mu$ satisfies \eqref{E:Dalang}, then
for any $t>0$, $x\in \bR^d$ and $n \geq 1$, $f_n(\cdot,t,x)\in \cH^{\otimes n}$ and $\|f_n(\cdot,t,x)\|_{\cH^{\otimes n}}^2=a_n(t,x)$,
where
\begin{align*}
a_n(t,x):=&\frac{1}{(2\pi)^{nd}}\int_{\bR^{nd}}\int_{[0,t]^{2n}} \prod_{j=1}^{n}\gamma(t_j-s_j)
\cF f_n(t_1,\cdot,\ldots,t_n,\cdot,t,x)(\xi_1,\ldots,\xi_n)\\
& \times \overline{\cF f_n(s_1,\cdot,\ldots,s_n,\cdot,t,x)(\xi_1,\ldots,\xi_n)} \ud t_1 \ldots \ud t_n \ud s_1 \ldots \ud s_n\mu(\ud\xi_1)\ldots \mu(\ud\xi_n).
\end{align*}
\end{theorem}

\begin{proof} We apply Theorem 2.10.c) of \cite{balan-song16}. To see that $f_n(\cdot,t,x)$ satisfies the conditions of this theorem, we note that
by Lemma \ref{L:Fourier-fn}, the map $$(t_1,\ldots,t_n,\xi_1, \ldots,\xi_n) \mapsto \cF f_n(t_1,\cdot,\ldots,t_n,\cdot,t,x)(\xi_1,\ldots,\xi_n)=:
\phi_{\xi_1,\ldots,\xi_n}(t_1,\ldots,t_n)$$ is measurable on $\bR^n \times \bR^{nd}$. Moreover, for any $\xi_1, \ldots,\xi_n \in \bR^d$, the map
$(t_1,\ldots,t_n) \mapsto \phi_{\xi_1,\ldots,\xi_n}(t_1,\ldots,t_n)$ is continuous and bounded by $\lambda^n J_+(t,x)$.  Calculations similar to those presented in Section \ref{S:Exist} below show that $a_n(t,x)<\infty$.
\end{proof}

For the remaining of the article, we assume that \eqref{E:Dalang} holds. By Theorems \ref{exist-sol-th} and \ref{T:fn-in-H}, it follows that the necessary and sufficient condition for the existence of the solution of \eqref{E:heat} is the following: for any $t \geq 0$ and $x \in \bR^d$,
\begin{equation}
\label{series-conv}
\sum_{n \geq 0}n! \, \|\widetilde{f}_n (\cdot,t,x)\|_{\cH^{\otimes n}}^2<\infty.
\end{equation}

\section{Existence of Solution} 
\label{S:Exist}

In this section, we give the proof of Theorem \ref{T:ExUni}. This will be based on several preliminary results.

In the first step, we will show that condition \eqref{series-conv} holds for any $t \geq 0$ and $x \in \bR^d$. When this condition holds,
$$\E(|u(t,x)|^2)=\sum_{n \geq 0} \E(|J_n(t,x)|^2)<\infty.$$
We denote
$$\E(|J_n(t,x)|^2)=\E(|I_n(f_n(\cdot,t,x))|^2)=n! \|\widetilde{f}_n(\cdot,t,x)\|_{\cH^{\otimes n}}^{2}=:\frac{1}{n!}\alpha_n(t,x).$$
With this notation, condition \eqref{series-conv} becomes:
\begin{equation}
\label{series-conv2}
\sum_{n \geq 0}\frac{1}{n!}\alpha_n(t,x)<\infty.
\end{equation}

Using the definition of the norm in $\cH^{\otimes n}$, we see that
$$\alpha_n(t,x)=\int_{[0,t]^{2n}}\prod_{j=1}^{n}\gamma(t_j-s_j)
\psi_{t,x}^{(n)}({
\bf t},{\bf s})\ud {\bf  t}\ud {\bf  s},$$
where
\begin{align}
\nonumber
\psi_{t,x}^{(n)}({\bf t},{\bf s})&=\int_{\bR^{2nd}}\prod_{j=1}^{n}f(x_j-y_j)g_{{\bf t},t,x}^{(n)}(x_1,\ldots,x_n)g_{{\bf s},t,x}^{(n)}(y_1, \ldots,y_n)\ud {\bf  x}\ud {\bf  y}\\
\label{E:def-psi}
&= \frac{1}{(2\pi)^{nd}} \int_{\bR^{nd}}\cF g_{{\bf t},t,x}^{(n)}(\xi_1,\ldots,\xi_n)\overline{\cF g_{{\bf s},t,x}^{(n)}(\xi_1, \ldots,\xi_n)}\mu(\ud\xi_1)\ldots \mu(\ud\xi_n)
\end{align}
and we denote
\begin{align*}
g_{{\bf t},t,x}^{(n)}(x_1,\ldots,x_n)&=n! \widetilde{f}_n(t_1,x_1,\ldots,t_n,x_n,t,x)\\
&=  \lambda^n \sum_{\rho \in S_n}G(t-t_{\rho(n)},x-x_{\rho(n)})\ldots G(t_{\rho(2)}-t_{\rho(1)},x_{\rho(2)}-x_{\rho(1)})\\
&\quad\qquad\times J_0(t_{\rho(1)},x_{\rho(1)}) 1_{\{0<t_{\rho(1)<\ldots<t_{\rho(n)}}<t\}}.
\end{align*}

To estimate $\alpha_n(t,x)$ we proceed as on page 11 of \cite{HHNT15}. By the Cauchy-Schwarz inequality and the inequality $ab \leq \frac{1}{2}(a^2+b^2)$, we have:
$$\psi_{t,x}^{(n)}({\bf t},{\bf s}) \leq \psi_{t,x}^{(n)}({\bf t},{\bf t})^{1/2}
\psi_{t,x}^{(n)}({\bf s},{\bf s})^{1/2} \leq \frac{1}{2}\left(\psi_{t,x}^{(n)}({\bf t},{\bf t})+\psi_{t,x}^{(n)}({\bf s},{\bf s})\right).$$

Since $\gamma$ is symmetric, we obtain:
\begin{align*}
\alpha_n(t,x)\leq & \frac{1}{2}\left(\int_{[0,t]^{2n}}\prod_{j=1}^{n}\gamma(t_j-s_j)\psi_{t,x}^{(n)}({\bf t},{\bf t})\ud {\bf  t}\ud {\bf  s}+\int_{[0,t]^{2n}}\prod_{j=1}^{n}\gamma(t_j-s_j)\psi_{t,x}^{(n)}({\bf s},{\bf s})\ud {\bf  t}\ud {\bf  s}\right) \\
=& \int_{[0,t]^{2n}}\prod_{j=1}^{n}\gamma(t_j-s_j)\psi_{t,x}^{(n)}({\bf t},{\bf t})\ud {\bf  t}\ud {\bf  s}.
\end{align*}

We now use the following elementary lemma, which can be proved by induction.

\begin{lemma}[Lemma 3.3 of \cite{balan-song16}]
\label{basic-ineq-lemma}
For any $n \geq 1$ and for any non-negative (or integrable) function $h :[0,t]^n \to \bR$, we have
\begin{equation}
\label{basic-ineq}
\int_{[0,t]^n}\prod_{j=1}^n\gamma(t_j-s_j)h({\bf t})\ud {\bf  t}\ud {\bf  s}\leq \Gamma_t^n \int_{[0,t]^n}h({\bf t})\ud {\bf  t},
\end{equation}
where $\Gamma_t$ is defined in \eqref{E:Gammat}.
 \end{lemma}


\vspace{3mm}

Applying Lemma \ref{basic-ineq-lemma} to the function $h({\bf t})=\psi_n({\bf t},{\bf t})$, we obtain:
\begin{equation}
\label{estimate-alpha1}
\alpha_n(t,x) \leq \Gamma_t^n \int_{[0,t]^n}\psi_{t,x}^{(n)}({\bf t},{\bf t})\ud {\bf  t}=\Gamma_t^n \sum_{\rho \in S_n}\int_{0<t_{\rho(1)}<\ldots<t_{\rho(n)}<t}\psi_{t,x}^{(n)}({\bf t},{\bf t})\ud {\bf  t}.
\end{equation}


\begin{lemma}
\label{le-lemma}
If $0<t_{\rho(1)}<\ldots<t_{\rho(n)}<t=:t_{\rho(n+1)}$, then
$$\psi_{t,x}^{(n)}({\bf t},{\bf t}) \leq  \frac{\lambda^{2n} J_+^2(t,x)}{(2\pi)^{nd}} \int_{\bR^{nd}}
\exp\left\{-\sum_{k=1}^{n}
\left(\frac{t_{\rho(k+1)}-t_{\rho(k)}}{t_{\rho(k+1)}t_{\rho(k)}}
\left|\sum_{j=1}^{k}t_{\rho(j)}\xi_{j}\right|^2\right) \right\}\mu(\ud\xi_1)\ldots \mu(\ud\xi_n).$$
\end{lemma}

\begin{proof}
By definition,
$$\psi_{t,x}^{(n)}({\bf t},{\bf t})=\frac{1}{(2\pi)^{nd}}
\int_{\R^{nd}}\left|\cF g_{{\bf t},t,x}^{(n)}(\xi_1,\ldots,\xi_n)\right|^2\mu(\ud \xi_1)\dots\mu(\ud \xi_n).$$
Similarly to Lemma \ref{L:Fourier-fn}, it can be shown that
\begin{align}
\nonumber
\cF g_{{\bf t},t,x}^{(n)}(\xi_1,\ldots,\xi_n)=&\lambda^n \prod_{k=1}^{n} \exp\left\{-\frac{1}{2}
\frac{t_{\rho(k+1)}-t_{\rho(k)}}{t_{\rho(k)} t_{\rho(k+1)}} \left|\sum_{j=1}^k t_{\rho(j)} \xi_{\rho(j)}\right|^2 \right\}
\exp\left\{- \frac{i}{t}\left(\sum_{j=1}^{n}t_{j} \xi_{j}\right)\cdot x \right\} \\
\label{E:Fourier-gt}
 &\times\int_{\bR^d} \exp\left\{-i \left[ \sum_{j=1}^{n}\left(1-\frac{t_{j}}{t}\right)\xi_{j} \right]\cdot x_0 \right\}G(t,x-x_0)u_0(\ud x_0).
\end{align}
The conclusion follows.
\end{proof}

\begin{lemma}
\label{L:Jn}
For any $t>0$ and $x \in \bR^d$,
$$\E\left(|J_n(t,x)|^2\right)=\frac{1}{n!}\alpha_n(t,x) \leq \lambda^{2n} \Gamma_t^n  J_+^2(t,x) \int_{0<t_{1}<\ldots<t_{n}<t}I_{t}^{(n)}(t_1,\ldots,t_n)\ud t_1\ldots \ud t_n,$$
where
$$I_{t}^{(n)}(t_1,\ldots,t_n):=\frac{1}{(2\pi)^{nd}}\int_{\bR^{nd}}
\exp\left\{-\sum_{k=1}^{n}
\left(\frac{t_{k+1}-t_{k}}{t_{k+1}t_{k}}
\left|\sum_{j=1}^{k}t_{j}\xi_j\right|^2\right) \right\}\mu(\ud\xi_1)\ldots \mu(\ud\xi_n)$$
and $t_{n+1}=t$.
\end{lemma}
\begin{proof}
Using Lemma \ref{le-lemma}, it follows that
\begin{align*}
  \int_{0<t_{\rho(1)}<\ldots<t_{\rho(n)}<t}\psi_{t,x}^{(n)}({\bf t},{\bf t})\ud {\bf t}
\leq  & \lambda^{2n} J_+^2(t,x)\frac{1}{(2\pi)^{nd}} \int_{0<t_{\rho(1)}<\ldots<t_{\rho(n)}<t}\ud {\bf t} \int_{\bR^{nd}}\mu(\ud\xi_1)\ldots \mu(\ud\xi_n)\\
&\qquad\times \exp\left\{-\sum_{k=1}^{n}
\left(\frac{t_{\rho(k+1)}-t_{\rho(k)}}{t_{\rho(k+1)}t_{\rho(k)}}
\left|\sum_{j=1}^{k}t_{\rho(j)}\xi_j\right|^2\right) \right\} \\
= &  \lambda^{2n} J_+^2(t,x) \frac{1}{(2\pi)^{nd}}\int_{0<t_{1}'<\ldots<t_{n}'<t}\ud {\bf  t}'\int_{\bR^{nd}}
\mu(\ud\xi_1)\ldots \mu(\ud\xi_n)\\
&\qquad\times\exp\left\{-\sum_{i=1}^{n}
\left(\frac{t_{k+1}'-t_{k}'}{t_{k+1}'t_{k}'}
\left|\sum_{j=1}^{k}t_{j}'\xi_j\right|^2\right) \right\},
\end{align*}
where for the last equality we used the change of variable $t_{j}'=t_{\rho(j)}$ for $j=1, \ldots,n$ and we denoted $t_{n+1}'=t$. The conclusion follows using \eqref{estimate-alpha1}.
\end{proof}

\vspace{3mm}

We now use the following maximum principle, which is of independent interest.

\begin{lemma}
\label{max-principle-lemma}
Let $\mu$ be a tempered measure on $\bR^d$ such that its Fourier transform in $\cS'_{\bC}(\bR^d)$ is a locally integrable function $f$, i.e. \eqref{E:f-Fourier-mu} holds. Assume that $f$ is non-negative. Then for any $\psi \in \cS(\bR^d)$ such that $\psi * \widetilde{\psi}$ is non-negative, where $\widetilde{\psi}(x)=\psi(-x)$ for all $x \in \bR^d$, we have:
\begin{equation}
\label{max-principle1}
\sup_{\eta \in \bR^d}\int_{\bR^d}|\cF \psi(\xi+\eta)|^2 \mu(\ud\xi)=\int_{\bR}|\cF \psi(\xi)|^2\mu(\ud\xi).
\end{equation}
In particular, for any $a>0$ and $t>0$,
\begin{equation}
\label{max-principle2}
\sup_{\eta \in \bR^d}\int_{\bR^d}e^{-a|t\xi+\eta|^2} \mu(\ud\xi)=\int_{\bR}e^{-a|t\xi|^2}\mu(\ud\xi).
\end{equation}
\end{lemma}

\begin{proof}
Note that for any function $g\in L^1(\bR^d)$ and for any $\xi,\eta \in \bR^d$, we have
$$\cF g(\xi+\eta)=\int_{\bR^d}e^{-i \xi \cdot x}e^{-i \eta \cdot x}g(x)\ud x= \cF (e^{-i \eta \cdot}g)(\xi).$$ Applying this to the function $g=\psi*\widetilde{\psi}$, we obtain that
$$|\cF \psi (\xi+\eta)|^2=\cF (\psi * \widetilde{\psi})(\xi+\eta)=\cF(e^{-i \eta \cdot }(\psi* \widetilde{\psi}))(\xi),$$
for any $\xi,\eta \in \bR^d$. For each $\eta \in \bR^d$ fixed, we apply  \eqref{E:f-Fourier-mu} to the function $\varphi=e^{-i \eta \cdot}(\psi *\widetilde{\psi}) \in \cS_{\bC}(\bR^d)$. We obtain that for any $\eta \in \bR^d$,
\begin{align*}
0 & \leq \int_{\bR^d}|\cF \psi (\xi+\eta)|^2\mu(\ud\xi)=\int_{\bR^d}
\cF(e^{-i \eta \cdot }(\psi* \widetilde{\psi}))(\xi)\mu(\ud\xi) \\
&=
(2\pi)^d \int_{\bR^d} e^{-i \eta \cdot x} (\psi* \widetilde{\psi})(x)f(x)\ud x
=(2\pi)^d \left|\int_{\bR^d} e^{-i \eta \cdot x} (\psi* \widetilde{\psi})(x)f(x)\ud x\right| \\
&\leq (2\pi)^d \int_{\bR^d} (\psi* \widetilde{\psi})(x)f(x)\ud x=\int_{\bR^d}|\cF \psi(\xi)|^2\mu(\ud\xi),
\end{align*}
using the fact that $|\int \ldots| \leq \int|\ldots|$ and $|e^{-i\eta \cdot x}|=1$. This proves \eqref{max-principle1}.

To prove \eqref{max-principle2}, we fix $t>0$ and consider the measure $\mu_t=\mu \circ h_t^{-1}$ on $\bR^d$, where $h_t(x)=t\xi$ for all $\xi \in \bR^d$. Using \eqref{E:Fourier-G}, it follows that
\[
\int_{\bR^d}e^{-a|t\xi+\eta|^2}\mu(\ud\xi)=
\int_{\bR^d}e^{-a|\xi+\eta|^2}\mu_t(\ud\xi)=\int_{\bR^d}|\cF G(a,\cdot)(\xi+\eta)|^2 \mu_t(\ud\xi).
\]
We note that $\mu_t$ is a tempered measure and its Fourier transform in $\cS'_{\bC}(\bR^d)$ is the locally integrable function $f_t$, defined by $f_t(x)=f(tx)$ for all $x \in \bR^d$ (see the proof of Lemma 3.2 of \cite{balan-song16}).
Applying \eqref{max-principle1} to the function $\psi=G(a,\cdot) \in \cS(\bR^d)$ and the Fourier pair $(\mu_t,f_t)$, we obtain that for any $\eta \in \bR^d$,
$$\int_{\bR^d}|\cF G(a,\cdot)(\xi+\eta)|^2\mu_t(\ud\xi)\leq \int_{\bR^d}|\cF G(a,\cdot)(\xi)|^2\mu_t(\ud\xi)=\int_{\bR^d}e^{-a|t\xi|^2}\mu(\ud\xi).$$
This completes the proof of Lemma \ref{max-principle-lemma}.
\end{proof}

Now we need to introduce some notation.
Following \cite{CK15Color}, define
\begin{align}
\label{E:k}
k(t):= \int_{\R^d}f(z)G(t,z)\ud z.
\end{align}
By \eqref{E:f-Fourier-mu} and \eqref{E:Fourier-G}, we see that
\begin{align}\label{E:k2}
k(t)=\frac{1}{(2\pi)^{d}}\int_{\R^d}\exp\left(-\frac{t|\xi|^2}{2}\right)\mu(\ud\xi),
\end{align}
from which one can see that $k(t)$ is a nonincreasing function.
By the dominated convergence theorem and condition \eqref{E:Dalang}, we see that $k$ is continuous on $(0,\infty)$.

Using Lemma \ref{max-principle-lemma} and the definition of the function $k$, we obtain the following estimate for the integral $I_{t}^{(n)}(t_1,\ldots,t_n)$ defined in Lemma \ref{L:Jn}.

\begin{lemma}\label{L:I-bnd}
For any $0<t_1<\ldots<t_n<t=:t_{n+1}$, we have that
\[
I_{t}^{(n)}(t_1,\ldots,t_n)\leq J_{t}^{(n)}(t_1,\ldots,t_n):=\prod_{i=1}^{n} k\left(\frac{2(t_{i+1}-t_i)t_i}{t_{i+1}}\right),
\]
and hence,
$$\int_{0<t_1<\dots<t_n<t} I_t^{(n)}(t_1,\dots,t_n) \ud t_1\dots \ud t_n \leq \int_{0<t_1<\dots<t_n<t} J_t^{(n)}(t_1,\dots,t_n) \ud t_1\dots \ud t_n.$$
\end{lemma}

\begin{proof}
We denote $a_i=(t_{i+1}-t_i)/(t_it_{i+1})$ for all $i=1, \ldots,n$ and we write
\begin{align*}
 I_{t}^{(n)}(t_1,\ldots,t_n)=&\frac{1}{(2\pi)^{nd}}\int_{\bR^d}\mu(\ud\xi_1) \: e^{-a_1|t_1\xi_1|^2}\int_{\bR^d}
 \mu(\ud\xi_2)\: e^{-a_2|t_1\xi_1+t_2\xi_2|^2} \ldots \\
&\times \int_{\bR^d}\mu(\ud\xi_n)\: e^{-a_n|t_1\xi_1+\ldots+t_n\xi_n|^2}.
\end{align*}
For the inner integral, we note that for any $\xi_1, \ldots,\xi_{n-1}\in \bR^d$,
$$\int_{\bR^d}e^{-a_n|t_1\xi_1+\ldots+t_n\xi_n|^2}\mu(\ud\xi_n)\leq \sup_{\eta \in \bR^d}\int_{\bR^d}e^{-a_n|\eta+t_n\xi_n|^2}\mu(\ud\xi_n)=
\int_{\bR^d}e^{-a_n|t_n\xi_n|^2}\mu(\ud\xi_n),$$
by Lemma \ref{max-principle-lemma}. The other integrals are estimated similarly. Hence,
\begin{equation}
\label{E:estim-I}
I_{t}^{(n)}(t_1,\ldots,t_n)\leq \frac{1}{(2\pi)^{nd}}\prod_{i=1}^{n}
\int_{\R^d}\exp\left(-\frac{t_{i+1}-t_i}{t_i t_{i+1}}|t_i\xi_i|^2\right)\mu(\ud\xi_i).
\end{equation}
The conclusion follows by the definition \eqref{E:k2} of the function $k(t)$.
\end{proof}

For $t\ge 0$, denote $h_0(t):=1$ and for $n\ge 1$
\begin{align*}
h_n(t):=\int_0^t h_{n-1}(s) k(t-s)\ud s.
\end{align*}

Note that $h_n(t) \in [0,\infty]$ for all $t \geq 0$ and $h_n$ is nondecreasing (by Lemma 2.6 of \cite{CK15Color}).
Moreover, under Dalang's condition \eqref{E:Dalang}, for any $\beta>0$ and for any integer $n \geq 0$,
\begin{equation}
\label{int-hn}
\int_0^{\infty}e^{-\beta t}h_n(t)\ud t=\frac{1}{\beta}\left(\int_0^{\infty}e^{-\beta t}k(t)\ud t \right)^{n}=\frac{1}{\beta}[2 \Upsilon(2\beta)]^n<\infty.
\end{equation}
Hence $h_n(t)<\infty$ for almost all $t \geq 0$. Since $h_n$ is non-decreasing, it follows that $h_n(t)<\infty$ for all $t \geq 0$.

\begin{lemma}\label{L:Int-I}
For any $t\ge 0$ and for any integer $n \geq 1$, it holds that
\[
\int_{0<t_1<\dots<t_n<t} J_t^{(n)}(t_1,\dots,t_n) \ud t_1\dots \ud t_n  \le
 2^n h_n(t).
\]
\end{lemma}

\begin{proof}
We first show that for any $n \geq 0$ and $t \geq 0$,
\begin{equation}
\label{hn-ineq}
\int_0^{t} k\left(\frac{2(t-s)s}{t}\right) h_n(s)\ud s \leq 2 h_{n+1}(t).
\end{equation}

Fix $n\ge 0$.
By symmetry, we see that for any $t\ge0$,
\begin{align*}
\int_0^{t} k\left(\frac{2(t-s)s}{t}\right) h_n(s)\ud s&=
\int_0^{t} k\left(\frac{2s(t-s)}{t}\right) h_n(t-s)\ud s\\
&\leq  \int_{0}^{t/2} k\left(\frac{2s(t-s)}{t}\right)  h_n(t-s) \ud s+
\int_{t/2}^{t} k\left(\frac{2s(t-s)}{t}\right)  h_n(s) \ud s\\
&=
2\int_{0}^{t/2} k\left(\frac{2s(t-s)}{t}\right)  h_n(t-s) \ud s,
\end{align*}
where for the inequality above we used the fact that $h_n$ is non-decreasing and hence $h_n(t-s) \leq h_n(s)$ for $s \geq t/2$.
Because $k(t)$ is nonincreasing and $2s(t-s)/t\ge s$ for $s\in [0,t/2]$, we have that
\begin{align*}
\int_{0}^{t/2} k\left(\frac{2s(t-s)}{t}\right)  h_n(t-s) \ud s
&\le
\int_0^{t/2} k\left(s\right)h_n(t-s)\ud s\le
\int_0^{t} k\left(s \right)h_n(t-s)\ud s=h_{n+1}(t).
\end{align*}
This proves \eqref{hn-ineq}.

Denote $I=\int_{0<t_1<\dots<t_n<t} J_t^{(n)} (t_1,\dots,t_n) \ud t_1\dots \ud t_n$.
By inequality \eqref{hn-ineq},
\begin{align*}
I=&\int_0^t \ud t_n \: k\left(\frac{2(t-t_n)t}{t_n} \right)
\int_0^{t_n}
\ud t_{n-1} \: k\left(\frac{2(t_n-t_{n-1})t}{t_{n-1}} \right)\dots
\\
&\times \dots \int_0^{t_2} \ud t_{1} \: k\left(\frac{2(t_2-t_1)t_2}{t_1}\right)\\
\le&
\: 2 \int_0^t \ud t_n \: k\left(\frac{2(t-t_n)t}{t_n}\right)
\int_0^{t_n}
\ud t_{n-1} \: k\left(\frac{2(t_n-t_{n-1})t}{t_{n-1}}\right)\dots\\
&\times \dots \int_0^{t_3} \ud t_{2} \: k\left(\frac{2(t_3-t_2)t_3}{t_2}\right) h_1(t_2) \leq \ldots \leq \\
\le &  2^{n-1}
\int_0^t \ud t_n \: k\left(t-t_n\right) h_{n-1}(t_n)\\
=& 2^n h_n(t).
\end{align*}
This proves Lemma \ref{L:Int-I}.
\end{proof}

The next lemma gives us more information about $h_n(t)$.

\begin{lemma}
If condition \eqref{E:Upsilon0} holds, then
\begin{equation}
\label{bound-hn}
h_n(t) \leq [2 \Upsilon(0)]^n \quad \mbox{for any $t>0$ and $n \geq 1$}.
\end{equation}
\end{lemma}

\begin{proof} By Fubini's theorem,
\[
h_1(t)=\frac{1}{(2\pi)^d}\int_0^t \int_{\bR^d}e^{-s|\xi|^2/2}\mu(\ud\xi)\ud s= \frac{1}{(2\pi)^d} \int_{\bR^d}\frac{1-e^{-t|\xi|^2/2}}{|\xi|^2/2}\mu(\ud\xi).
\]
Hence, by the monotone convergence theorem,
\[
\lim_{t\rightarrow\infty}h_1(t) =
 \frac{1}{(2\pi)^d} \int_{\bR^d}\frac{\mu(\ud\xi)}{|\xi|^2/2} = 2 \Upsilon (0). 
\]
Because $h_1(t)$ is nondecreasing, the above limit shows that
$h_1(t)\le 2\Upsilon(0)$. The conclusion follows by induction on $n$.
\end{proof}

We need to introduce some additional notation.
For $\gamma\ge 0$ and $t\ge 0$, define
\begin{align}\label{E:H}
 H(t;\gamma):= \sum_{n=0}^\infty \gamma^n h_n(t)\quad
\text{and}\quad
 \widetilde{H}(t;\gamma):= \sum_{n=0}^\infty  \sqrt{\gamma^n h_n(t)}\: .
\end{align}
Note that $H(t;\gamma)\in [0,\infty]$ and $\widetilde{H}(t;\gamma)\in [0,\infty]$  for all $t \geq 0$ and $\gamma \geq 0$. Since $h_n$ is non-decreasing for all $n$,
both $t\mapsto H(t;\gamma)$ and $t\mapsto \widetilde{H}(t;\gamma)$ are nondecreasing.

\begin{lemma}
\label{L:Variation}
For any $t \geq 0$ and $\gamma > 0$, $H(t;\gamma)<\infty$ and $\widetilde{H}(t;\gamma)<\infty$.
For all $\gamma> 0$,
\[
\limsup_{t\rightarrow\infty}\frac{1}{t}\log H(t;\gamma)\le \theta \quad
\text{and}\quad
\limsup_{t\rightarrow\infty}\frac{1}{t}\log \widetilde{H}(t;\gamma)\le \theta,
\]
where this constant $\theta$ can be chosen as
\begin{align}\label{E:theta}
\theta:=\theta(\gamma)=\inf\left\{\beta>0: \: \Upsilon(2\beta) < \frac{1}{2\gamma}\: \right\}.
\end{align}
Moreover, if $\Upsilon(0)<\infty$, then for all $t\ge 0$ and $0<\gamma<1/[2\Upsilon(0)]$,
\[
H(t;\gamma)\le \frac{1}{1-2\gamma \Upsilon(0)}\quad\text{and}\quad
\widetilde{H}(t;\gamma)\le \frac{1}{1-\sqrt{2\gamma \Upsilon(0)}}\:.
\]
\end{lemma}
\begin{proof}
The statements for $H(t;\gamma)$ are proved in Lemma 2.5 in \cite{CK15Color} (with $\nu=1$). We include the argument for the sake of completeness.
Let $\gamma>0$ be arbitrary. By the dominated convergence theorem, $\lim_{\beta \to \infty}\Upsilon(\beta)=0$, Hence, there exists $\beta>0$ such that $2\Upsilon\left(2\beta\right)\gamma <1$. By \eqref{int-hn}, we have:
\begin{align}\label{E:eHsum}
\int_0^{\infty}e^{-\beta t}H(t;\gamma)\ud t=\sum_{n \geq 0}\gamma^n \int_0^{\infty}e^{-\beta t}h_n(t)\ud t=\sum_{n \geq 0} \gamma^n [2\Upsilon(2\beta)]^n<\infty.
\end{align}
Hence $H(t;\gamma)<\infty$ for almost all $t \geq 0$. Since $t \mapsto H(t;\gamma)$ is non-decreasing, it follows that $H(t;\gamma)<\infty$ for all $t \geq 0$.
By Lemma \ref{L:LimSupHt} (Appendix \ref{A:lemmas}), we conclude that
\[
\limsup_{t \to \infty}\frac{1}{t}\log H(t;\gamma) \leq \inf \left\{\beta>0; \int_0^{\infty}e^{-\beta t}H(t;\gamma)dt<\infty\right\}
=\theta(\gamma),
\]
where $\theta(\gamma)$ is defined in \eqref{E:theta} and the last inequality is due to the fact that (thanks to \eqref{E:eHsum})
\[
\int_0^{\infty}e^{-\beta t}H(t;\gamma)dt<\infty\quad\Longleftrightarrow\quad
2\gamma\Upsilon(2\beta)<1.
\]

The results for $\widetilde{H}(t;\gamma)$ are proved similarly.
Notice that, due to the Cauchy-Schwarz inequality and \eqref{int-hn}, for any $\beta>0$,
\begin{align*}
\int_{\R_+} e^{-\beta t} \sqrt{h_n(t)}\ud t
&\le \frac{1}{\sqrt{\beta}} \left(\int_{\R_+} e^{-\beta t} h_n(t)\ud t\right)^{1/2}\\
&=\frac{1}{\beta}\left(\int_{\R_+}e^{-\beta t} k(t) \ud t\right)^{n/2}
=\frac{[2\Upsilon(2\beta)]^{n/2}}{\beta}.
\end{align*}

Therefore, for $\beta>0$ such that $2\Upsilon(2\beta)\gamma<1$, we have
$$\int_{0}^{\infty}e^{-\beta t}\widetilde{H}(t;\gamma)dt=\sum_{n \geq 0}\gamma^{n/2}\int_0^{\infty}e^{-\beta t}\sqrt{h_n(t)} \ud t
=\frac{1}{\beta}\sum_{n \geq 0}\gamma^{n/2}[2\Upsilon(2\beta)]^{n/2}<\infty.$$
Using the same argument as above, we infer that $\widetilde{H}(t;\gamma)<\infty$ for any $t \geq 0$ and $\gamma>0$.
By Lemma \ref{L:LimSupHt} (Appendix \ref{A:lemmas}), we conclude that
\[
\limsup_{t \to \infty}\frac{1}{t}\log \widetilde{H}(t;\gamma) \leq \theta(\gamma).
\]

When $\Upsilon(0)<\infty$, using \eqref{bound-hn}, we have
\begin{align*}
H(t;\gamma) &\leq \sum_{n \geq 0}\gamma^{n}[2\Upsilon(0)]^{n}=\frac{1}{1-2\gamma \Upsilon(0)}
\quad\text{and}\\
\widetilde{H}(t;\gamma) &\leq \sum_{n \geq 0}\gamma^{n/2}[2\Upsilon(0)]^{n/2}=\frac{1}{1-\sqrt{2\gamma \Upsilon(0)}},
\end{align*}
whenever $2\gamma \Upsilon(0)<1$.
This completes the proof of Lemma \ref{L:Variation}.
\end{proof}

\bigskip
We are now ready to prove Theorem \ref{T:ExUni}.

\begin{proof}[Proof of Theorem \ref{T:ExUni}]
(a) We first show the existence (and uniqueness) of the solution. As mentioned at the beginning of this section, this reduces to showing that condition \eqref{series-conv2} holds for any $t \geq 0$ and $x \in \bR^d$. By Lemma \ref{L:Jn}, Lemma \ref{L:I-bnd} and Lemma \ref{L:Int-I},
\[
\E\left(|J_n(t,x)|^2\right)=\frac{1}{n!}\alpha_n(t,x)\leq \lambda^{2n}J_+^2(t,x)\Gamma_t^n 2^n h_n(t).
\]
Hence, by invoking Lemma \ref{L:Variation}, we see that:
$$\sum_{n \geq 0}\frac{1}{n!}\alpha_n(t,x)\leq J_+^2(t,x)\sum_{n \geq 0}\lambda^{2n}
2^n h_n(t)=J_+^2(t,x) H(t; 2 \lambda^2 \Gamma_t)<\infty.$$
This concludes the proof of \eqref{series-conv2}.

Next, we prove \eqref{E:u-p}. Let $p \geq 2$ be arbitrary. Recall that we denote by $\|\cdot\|_p$ the $L^p(\Omega)$-norm. Since the norms $\|\cdot\|_p$ are equivalent on a fixed Wiener chaos $\cH_n$ (see, e.g., \cite[Theorem 5.10]{Janson97}),
\begin{equation}
\label{E:Jn}
\|J_n(t,x)\|_p \leq (p-1)^{n/2}\|J_n(t,x)\|_2 \leq (p-1)^{n/2}|\lambda|^{n}J_+(t,x)\Gamma_t^{n/2} 2^{n/2} \sqrt{h_n(t)}.
\end{equation}

By Minkowski's inequality,
\begin{align}
\label{E:Lp-L2}
\|u(t,x)\|_p & \le \sum_{n \geq 0}\|J_n(t,x)\|_p
  \leq
J_+(t,x)
\sum_{n \geq 0}|\lambda|^n(p-1)^{n/2}
\Gamma_t^{n/2}  2^{n/2} \sqrt{h_n(t)}\\
&= J_+(t,x)\widetilde{H}\left(t;2\lambda^2(p-1)\Gamma_t \right).
\label{E:UpMnt}
\end{align}
By Lemma \ref{L:Variation}, the previous quantity is finite. This concludes the proof of \eqref{E:u-p}.
Relation \eqref{E:sup-Ka} follows since the function $\widetilde{H}(t;\gamma)$ is non-decreasing in $t$ and $\gamma$ and
\begin{equation}
\label{E:sup-J0}
D_a:=\sup_{(t,x) \in K_a}J_{+}(t,x)<\infty.
\end{equation}
Note that \eqref{E:sup-J0} is a consequence of Lemma \ref{L:J0-cont} (Appendix \ref{A:cont-Jn}).

We now prove that $u$ is $L^p(\Omega)$-continuous on $(0,\infty) \times \bR^d$. Note that by Lemmas \ref{L:J0-cont} and \ref{L:Jn-cont} (Appendix \ref{A:cont-Jn}), $u_n=\sum_{k=0}^{n}J_k$ is $L^p(\Omega)$-continuous on $(0,\infty) \times \bR^d$. Let $a>0$ be arbitrary. By relation \eqref{E:Jn}, and the fact that $\Gamma_t$ and $h_n(t)$ are non-decreasing in $t$, we have
\begin{align*}
\sum_{n \geq 0}\sup_{(t,x) \in K_a} \|J_n(t,x)\|_p \leq & D_a \sum_{n \geq 0} (p-1)^{n/2} \Gamma_a^n |\lambda|^n 2^{n/2} \sqrt{h_n(a)}\\
=& D_a \widetilde{H}(a; 2\lambda^2 (p-1)\Gamma_a)<\infty,
\end{align*}
which means that $u_n(t,x) \to u(t,x)$ in $L^p(\Omega)$, uniformly on $K_a$. Hence $u$ is $L^p(\Omega)$-continuous on $K_a$. Since $a>0$ was arbitrary, $u$ is $L^p(\Omega)$-continuous on $(0,\infty) \times \bR^d$.

\bigskip
(b) Since $\Gamma_t \leq \Gamma_{\infty}$ for any $t>0$ and $\widetilde{H}(t;\gamma)$ is non-decreasing in $\gamma$, by \eqref{E:UpMnt} we have that:
\begin{equation}
\label{E:bound-mom-p}
\|u(t,x)\|_p \leq J_+(t,x) \widetilde{H}\left(t;2\lambda^2(p-1)\Gamma_{\infty}  \right).
\end{equation}
The conclusion follows by Lemma \ref{L:Variation}, using the fact that:
\begin{equation}
\label{E:log-J0}
\limsup_{t \to \infty}\frac{1}{t}\log J_+(t,x)=0.
\end{equation}
Note that \eqref{E:log-J0} is a consequence of \eqref{E:expo-u0}. For a proof of this, see page 19 of \cite{CK15Color}.

(c) It is shown in \cite{CK15Color} that $\Upsilon(0)<\infty$
happens only when $d\ge 3$. Let $\gamma=2\lambda^2(p-1)\Gamma_{\infty}$.
By \eqref{E:bound-mom-p} and Lemma \ref{L:Variation},
$$\|u(t,x)\|_p \leq \frac{J_+(t,x)}{1-\sqrt{2\gamma \Upsilon(0)}},$$
provided that $\gamma<1/[2\Upsilon(0)]$.
This last condition is equivalent to
$$|\lambda| \le \left[\frac{1}{4(p-1)\Gamma_\infty \Upsilon(0)}\right]^{1/2}=:\lambda_c.$$
The conclusion follows from \eqref{E:log-J0}.
\end{proof}

\section{The Riesz kernel case}\label{S:Riesz}

In this part, we prove Theorem \ref{T:Riesz}. For this, we build upon our previous estimate \eqref{E:estim-I} for the integral $I_t^{(n)}(t_1, \ldots,t_n)$, using the specific form of the measure $\mu$.
A similar argument can also be found in Example A.1 of \cite{CK15Color}.

In this section, we assume that $f$ is the Riesz kernel of order $0<\alpha<d$, given by:
\begin{equation}
\label{E:def-Riesz}
f(x)=\pi^{-d/2} 2^{-\alpha} \frac{\Gamma(\frac{d-\alpha}{2})}{\Gamma(\frac{\alpha}{2})}|x|^{-\alpha},
\end{equation}
Suppose that $\alpha<2$, so that \eqref{E:Dalang} holds.
It is known that $$\mu(\ud\xi)=|\xi|^{-(d-\alpha)}\ud\xi.$$
Hence, for any $t>0$,
$$\int_{\bR^d}e^{-t|\xi|^2}\mu(\ud\xi)=C_{\alpha,d}^{(1)}t^{-\alpha/2},$$
where $C_{\alpha,d}^{(1)}=\int_{\bR^d}e^{-|\xi|^2}|\xi|^{-(d-\alpha)}\ud\xi$. (This follows by the change of variable $\xi'=t^{1/2} \xi$.)
Using \eqref{E:estim-I}, it follows that
\begin{align}
\nonumber
I_{t}^{(n)}(t_1,\ldots,t_n)&\leq  C_{\alpha,d}^n \left(\frac{t_2-t_1}{t_1 t_2}\cdot \frac{t_3-t_2}{t_2 t_3} \cdot \ldots \cdot \frac{t-t_n}{t_n t}\,t_1^2 t_2^2 \ldots t_n^2 \right)^{-\alpha/2}\\
\label{calcul-I-t}
&= C_{\alpha,d}^n t^{\alpha/2}[t_1 (t_2-t_1)(t_3-t_2)\ldots (t-t_n)]^{-\alpha/2},
\end{align}
where $C_{\alpha,d}=(2\pi)^{-d}C_{\alpha,d}^{(1)}$.
We need the following elementary result.

\begin{lemma}
\label{int-h-lemma}
For any $h>-1$, we have
$$\int_{0<t_1<\ldots<t_n<t}[t_1 (t_2-t_1)(t_3-t_2)\ldots (t-t_n)]^{h}\ud t_1 \ldots \ud t_n=\frac{\Gamma(h+1)^{n+1}}{\Gamma((n+1)(h+1))}t^{n(h+1)+h}.$$
\end{lemma}

\begin{proof}
We write the integral as an iterated integral of the form:
$$\int_0^t (t-t_n)^h \left( \int_{0}^{t_n}(t_n-t_{n-1})^h \ldots \left( \int_0^{t_2}t_1^h (1-t_1)^{h}\ud t_1\right) \ldots \ud t_{n-1}\right)\ud t_n.$$
The inner integral is equal to $B(h+1,h+1)t_2^{2h+1}$, where $B(a,b)=\Gamma(a)\Gamma(b)/\Gamma(a+b)$ is the beta function. The second inner integral is $B(h+1,h+1)B(2(h+1),h+1)t_3^{3h+2}$. We continue in this manner. After $n$ steps, we obtain that the last integral is equal to
$$B(h+1,h+1)B(2(h+1),h+1)\ldots B(n(h+1),h+1)t^{(n+1)h+n}.$$
The conclusion follows from the definition of the beta function.
\end{proof}

Recall that the two-parameter {\it Mittag-Leffler
function} \cite[Section 1.2]{Podlubny99FDE} is defined as follows:
\begin{align}\label{E:Mittag-Leffler}
E_{\alpha,\beta}(z) := \sum_{k=0}^{\infty} \frac{z^k}{\Gamma(\alpha k+\beta)},
\qquad \alpha>0,\;\beta> 0.
\end{align}

\begin{lemma}[Theorem 1.3 p.~32 in \cite{Podlubny99FDE}]\label{L:Eab}
If $0<\alpha<2$, $\beta$ is an arbitrary complex number and $\mu$ is an
arbitrary real number such that
\[
\pi\alpha/2<\mu<\pi \wedge (\pi\alpha)\;,
\]
then for an arbitrary integer $p\ge 1$ the following expression holds:
\[
E_{\alpha,\beta}(z) = \frac{1}{\alpha} z^{(1-\beta)/\alpha}
\exp\left(z^{1/\alpha}\right)
-\sum_{k=1}^p \frac{z^{-k}}{\Gamma(\beta-\alpha k)} + O\left(|z|^{-1-p}\right)
,\quad |z|\rightarrow\infty,\quad |\arg(z)|\le \mu\:.
\]
\end{lemma}

\bigskip
\begin{proof}[Proof of Theorem \ref{T:Riesz}]
Using Lemma \ref{L:Jn}, relation \eqref{calcul-I-t} and Lemma \ref{int-h-lemma}, we have
\begin{align*}
\Norm{u(t,x)}_2^2&=\sum_{n\ge 0}\E\left[J_n^2(t,x)\right]\\
& \leq  \sum_{n\ge 0}\lambda^{2n}\Gamma_t^n J_+^2(t,x)C_{\alpha,d}^n t^{\alpha/2}\int_{0<t_1<\ldots<t_n<t}
[t_1(t_2-t_1)\ldots(t-t_n)]^{-\alpha/2}\ud t_1 \ldots \ud t_n \\
&= \sum_{n\ge 0}\lambda^{2n}\Gamma_t^n J_+^2(t,x) C_{\alpha,d}^n \frac{\Gamma(1-\alpha/2)^{n+1}}{\Gamma((n+1)(1-\alpha/2))}t^{n(1-\alpha/2)}\\
&= J_+^2(t,x)\Gamma(1-\alpha/2) E_{1-\alpha/2,1-\alpha/2}\left(\lambda^2\Gamma_t C_{\alpha,d} \Gamma(1-\alpha/2) t^{1-\alpha/2}\right)\\
&\le C_{\alpha,d}'  J_+^2(t,x) \exp\left(C_{\alpha,d}'' |\lambda|^{4/(2-\alpha)}\Gamma_t^{2/(2-\alpha)}t\right),
\end{align*}
where in the last step we have applied Lemma \ref{L:Eab} (see the proof of Proposition 3.2 of \cite{ChenDalang14FracHeat}
for a similar argument).
The constant $C_{\alpha,d}''$ can be any constant that is strictly bigger than
$[C_{\alpha,d}\Gamma(1-\alpha/2)]^{2/(2-\alpha)}$ and the constant $C_{\alpha,d}'$ is defined as
\begin{align*}
C_{\alpha,d}' &= \Gamma(1-\alpha/2)\:  \sup_{t\ge 0}
\frac{E_{1-\alpha/2,1-\alpha/2}\left(\lambda^{2}\Gamma_t C_{\alpha,d} \Gamma(1-\alpha/2) t^{1-\alpha/2}\right)}{\exp\left(C_{\alpha,d}'' |\lambda|^{4/(2-\alpha)}\Gamma_t^{2/(2-\alpha)}t\right)}\\
&=\Gamma(1-\alpha/2)\:  \sup_{x\ge 0}
\frac{E_{1-\alpha/2,1-\alpha/2}\left(C_{\alpha,d} \Gamma(1-\alpha/2) x^{1-\alpha/2}\right)}{\exp\left(C_{\alpha,d}'' x\right)}<\infty.
\end{align*}
As for the $p$-th moment, by Poicar\'e-type expansions of Gamma function (see \cite[5.11.3 o p. 140]{NIST2010}),
\[
\lim_{x\rightarrow\infty}\frac{\Gamma(x/2)}{\sqrt{\Gamma(x)}} =0.
\]
Hence, for some constant $K_\alpha>0$, $\Gamma((n+1)(1-\alpha/2))^{-1/2} \le K_\alpha \Gamma((n+1)(2-\alpha)/4)^{-1}$ for all $n\ge 0$.
Therefore, from \eqref{E:Lp-L2} and the calculations above, we see that
\begin{align*}
\Norm{u(t,x)}_p
&\le \sum_{n\ge 0}|\lambda|^n (p-1)^{n/2}\Gamma_t^{n/2} J_+(t,x)  C_{\alpha,d}^{n/2}
\frac{\Gamma(1-\alpha/2)^{(n+1)/2}}{\sqrt{\Gamma((n+1)(1-\alpha/2))}}t^{n(2-\alpha)/4}\\
&\le K_\alpha \sum_{n\ge 0}|\lambda|^n (p-1)^{n/2}\Gamma_t^{n/2} J_+(t,x) C_{\alpha,d}^{n/2} \frac{\Gamma(1-\alpha/2)^{(n+1)/2}}{\Gamma((n+1)(2-\alpha)/4)}t^{n(2-\alpha)/4}\\
&= K_\alpha \Gamma(1-\alpha/2)^{1/2}J_+(t,x)
E_{(2-\alpha)/4,(2-\alpha)/4}\left(|\lambda|\sqrt{p\Gamma_t C_{\alpha,d}\Gamma(1-\alpha/2)}\: t^{(2-\alpha)/4}\right)\\
&\le
\widetilde{C}_{\alpha,d}'  J_+(t,x) \exp\left(\widetilde{C}_{\alpha,d}''\: |\lambda|^{4/(2-\alpha)} p^{2/(2-\alpha)} \Gamma_t^{2/(2-\alpha)}t\right),
\end{align*}
where, by the same arguments as above, the constant
$\widetilde{C}_{\alpha,d}''$ is any constant that is strictly bigger than
$[C_{\alpha,d}\Gamma(1-\alpha/2)]^{2/(2-\alpha)}$ and $\widetilde{C}_{\alpha,d}'$ is defined as
\begin{align*}
\widetilde{C}_{\alpha,d}' &= K_\alpha \Gamma(1-\alpha/2)^{1/2}\:  \sup_{t\ge 0}
\frac{E_{(2-\alpha)/4,(2-\alpha)/4}\left(|\lambda|\sqrt{p \Gamma_t C_{\alpha,d} \Gamma(1-\alpha/2)}\: t^{(2-\alpha)/4}\right)}{\exp\left(\widetilde{C}_{\alpha,d}''\:|\lambda|^{4/(2-\alpha)} p^{2/(2-\alpha)} \Gamma_t^{2/(2-\alpha)}t\right)}\\
&=
K_\alpha \Gamma(1-\alpha/2)^{1/2}\:  \sup_{x\ge 0}
\frac{E_{(2-\alpha)/4,(2-\alpha)/4}\left(\sqrt{C_{\alpha,d} \Gamma(1-\alpha/2)}\: x\right)}{\exp\left(\widetilde{C}_{\alpha,d}''\: x^{4/(2-\alpha)}\right)}
<\infty,
\end{align*}
which does not depend on $p$.
This completes the proof of Theorem 3.9.
\end{proof}

\section{H\"older continuity}
\label{S:Holder}

In this section, we give the proof of Theorem \ref{T:Holder}. For this, we need a preliminary result.

\begin{lemma}
\label{L:EquvHoldCond}
Let $\beta\in (0,1)$ be arbitrary. Then
$$\int_{\R^d}\frac{\mu(\ud \xi)}{\left(1+|\xi|^2\right)^\beta}<\infty \quad
\mbox{if and only if} \quad \int_0^1 \frac{k(s)}{s^{1-\beta}} \ud s <\infty,$$
where the function $k$ is defined by \eqref{E:k}.
\end{lemma}

\begin{proof}
In the proof, we use $c$ and $C$ to denote general constants whose values may be different at each occurrence.
By the definition of $k(s)$, we have that
\begin{align*}
\int_0^1 s^{-(1-\beta)} k(s)\ud s &=
C \int_{\R^d}\mu(\ud \xi)\int_0^1 s^{-(1-\beta)} \exp\left(-\frac{s}{2}|\xi|^2\right)\ud s\\
&=C
\int_{\R^d}\mu(\ud \xi)\: \left(\frac{|\xi|^2}{2}\right)^{-\beta}
\int_0^{\frac{|\xi|^2}{2}} s^{-(1-\beta)} e^{-s} \ud s.
\end{align*}
Denote $g(x):= x^{-\beta}\int_0^{x} s^{-(1-\beta)} e^{-s} \ud s$.
Because
\[
\lim_{x\rightarrow 0_+}g(x)(1+x)^\beta = 1/\beta
\quad\text{and}\quad
\lim_{x\rightarrow \infty} g(x)(1+x)^\beta = \Gamma(\beta),
\]
both functions $g(x)(1+x)^\beta$ and $\left[g(x)(1+x)^\beta\right]^{-1}$
are continuous functions on $[0,\infty]$. Hence,
\[
\frac{c}{(1+x)^{\beta}} \le g(x) \le \frac{C}{(1+x)^{\beta}} \qquad\text{for all $x\ge 0$}.
\]
Therefore,
\[
c \int_{\R^d}\frac{\mu(\ud \xi)}{\left(1+|\xi|^2\right)^\beta}
\le \int_0^1 s^{-(1-\beta)} k(s)\ud s
\le C \int_{\R^d}\frac{\mu(\ud \xi)}{\left(1+|\xi|^2\right)^\beta}.
\]
This completes the proof of Lemma \ref{L:EquvHoldCond}.
\end{proof}

\begin{remark}
{\rm Let $f$ be the Riesz kernel with $\alpha\in (0,d \wedge 2)$ given by \eqref{E:def-Riesz}. Example A.1 of \cite{CK15Color} shows that in this case, $k(t)= C t^{-\alpha/2}$.
In this case, $\int_0^1 \frac{k(s)}{s^{1-\beta}} \ud s <\infty$
for all $\beta\in (\alpha/2,1)$.}
\end{remark}

\bigskip
\begin{proof}[Proof of Theorem \ref{T:Holder}]: We proceed as in the proof of Theorem 4.3 of \cite{balan-song16}, using the bounds obtained in the proof of Lemma \ref{L:Jn-cont} (Appendix \ref{A:cont-Jn}). Let $\theta=1-\beta$.

\vspace{3mm}

{\em Step 1. (left increments in time)} Let $(t,x),(t',x) \in K_a$. Say $t'=t-h$ for some $h>0$. We have
\begin{eqnarray}
\nonumber
\|u(t-h,x)-u(t,x)\|_p &\leq & \sum_{n\geq 0}(p-1)^{n/2} \|J_n(t-h,x))-J_n(t,x)\|_2 \\
\label{moment-time}
&\leq &\sum_{n\geq 0} (p-1)^{n/2} \left(\frac2{n!}\left[ A_n'(t,h,x)+B_n'(t,h,x)\right] \right)^{1/2},
\end{eqnarray}
where $A_n'(t,h,x)$ and $B_n'(t,h,x)$ are given by \eqref{def-A'}, respectively \eqref{def-B'}.

To find an upper bound for $A'_n(t,h,x)$, we use \eqref{E:bound-An}.
Notice that
\begin{align*}
 \exp\left(-\frac{t-h-t_n}{(t-h)t_n}\left|\sum_{j=1}^n t_j \xi_j\right|^2\right)
&\left[1-\exp\left(-\frac{h}{2t(t-h)}\left|\sum_{j=1}^n t_j \xi_j\right|^2\right)\right]^{\theta}\\
\le&
\exp\left(-\frac{t-h-t_n}{(t-h)t_n}\left|\sum_{j=1}^n t_j \xi_j\right|^2\right)
\left(\frac{h}{2t(t-h)}\right)^\theta\left|\sum_{j=1}^n t_j \xi_j\right|^{2\theta} \\
\le&
\frac{h^\theta}{t^{2\theta}} \exp\left(-\frac{t-h-t_n}{(t-h)t_n}\left|\sum_{j=1}^n t_j \xi_j\right|^2\right)
\left|\sum_{j=1}^n t_j \xi_j\right|^{2\theta}.
\end{align*}

Now for $A>0$ and $x\ge 0$, we see that
\[
\exp\left(-\frac{A}{2} x^2\right) x^{2\theta}
=\exp\left(-\frac{A}{2} x^2 + 2\theta \log x\right)
\le (2\theta/e)^\theta A^{-\theta}.
\]
This can be seen by noticing that the function $f(x)=-\frac{A}{2}x^2 +2\theta \log x, x>0$ attains its maximum at $x_0=\sqrt{2\theta/A}$.
Therefore, for some constant $C_\theta>0$ depending on $\theta$,
\begin{equation}
\label{E:expo-A}
\exp\left(-A x^2\right) x^{2\theta}\le
C_\theta A^{-\theta} \exp\left(-\frac{A}{2} x^2\right),\quad\text{for all $x\ge 0$.}
\end{equation}

Hence, this inequality implies that
\begin{align*}
\exp\left(-\frac{t-h-t_n}{(t-h)t_n}\left|\sum_{j=1}^n t_j \xi_j\right|^2\right)
\left|\sum_{j=1}^n t_j \xi_j\right|^{2\theta}
& \le C_\theta
\left(\frac{(t-h)t_n}{t-h-t_n}\right)^\theta
\exp\left(-\frac{t-h-t_n}{2(t-h)t_n}\left|\sum_{j=1}^n t_j \xi_j\right|^2\right).
\end{align*}
Thus, by denoting $t_{n+1}:=t-h$, and using \eqref{E:bound-An}, we have that
\begin{align*}
 A_n'(t,h,x) \le  & \Gamma_t^n \lambda^{2n} J_{+}^2(t,x) n!\frac{1}{(2\pi)^{nd}}
\int_{0<t_1<\ldots<t_n<t-h} \int_{\bR^{nd}} \prod_{k=1}^{n}\exp \left(-\frac{t_{k+1}-t_k}{2t_k t_{k+1}} \left|\sum_{j=1}^{k}t_j \xi_j \right|^2\right)\\
& \times \frac{C_\theta\: h^\theta}{(t-h-t_n)^\theta}
 \mu(\ud \xi_1) \ldots \mu(\ud \xi_n) \ud t_1 \ldots \ud t_n.
\end{align*}

Using \eqref{E:half-tn}, we have
\begin{align*}
& \hspace{-3em}\int_{0<t_1<\ldots<t_n<t-h} (t-h-t_n)^{-\theta} \prod_{k=1}^{n-1}
\exp\left(-\frac{1}{2} \frac{t_{k+1}-t_k}{t_{k+1}t_k}\left|\sum_{j=1}^k t_j \xi_j\right|^2\right) \\
& \qquad \times \exp\left(-\frac{1}{2} \frac{t-h-t_n}{(t-h)t_n}\left|\sum_{j=1}^n t_j \xi_j\right|^2\right) \ud t_1 \ldots \ud t_n \\
=  & 2^{-\theta}\int_{0<t_1<\ldots<t_n<t-h} \left(\frac{t-h}{2}-\frac{t_n}{2}\right)^{-\theta} \prod_{k=1}^{n-1}
\exp\left(-\frac{\frac{t_{k+1}}{2}-\frac{t_k}{2}}{\frac{t_{k+1}}{2}\frac{t_k}{2}}
\left|\sum_{j=1}^k \frac{t_j}{2} \xi_j\right|^2\right) \\
&  \qquad \times \exp\left(- \frac{\frac{t-h}{2}-\frac{t_n}{2}}{\frac{t-h}{2}\frac{t_n}{2}}\left|\sum_{j=1}^n
\frac{t_j}{2} \xi_j\right|^2\right) \ud t_1 \ldots \ud t_n\\
=  & 2^{-\theta} 2^n \int_{0<t_1'<\ldots<t_n'<\frac{t-h}{2}} \left(\frac{t-h}{2}-t_n'\right)^{-\theta} \prod_{k=1}^{n-1}
\exp\left(-\frac{t_{k+1}'-t_k'}{t_{k+1}'t_k'}
\left|\sum_{j=1}^k t_j' \xi_j\right|^2\right) \\
&  \qquad \times \exp\left(- \frac{\frac{t-h}{2}-t_n'}{\frac{t-h}{2}t_n'}\left|\sum_{j=1}^n
t_j'\xi_j\right|^2\right) \ud t_1' \ldots \ud t_n',
\end{align*}
where the second equality follows from the change of variables $t_k'=t_k/2$ for $k=1,\ldots,n$. Using the notation $t_{n+1}=\frac{t-h}{2}$, we see that
\begin{align*}
 & A_n'(t,h,x)\\
 \le  & \Gamma_t^n \lambda^{2n} J_{+}^2(t,x) n! 2^{-\theta}\frac{1}{(2\pi)^{nd}}
\int_{0<t_1<\ldots<t_n<\frac{t-h}{2}} \int_{\bR^{nd}} \prod_{k=1}^{n}\exp \left(-\frac{t_{k+1}-t_k}{t_k t_{k+1}} \left|\sum_{j=1}^{k}t_j \xi_j \right|^2\right)\\
& \times 2^n C_\theta\: h^\theta \left(\frac{t-h}{2}-t_n\right)^{-\theta}
 \mu(\ud \xi_1) \ldots \mu(\ud \xi_n) \ud t_1 \ldots \ud t_n\\
 =&
\Gamma_t^n \lambda^{2n} J_{+}^2(t,x) n! 2^{-\theta} 2^n C_\theta h^\theta
\int_{0<t_1<\ldots<t_n<\frac{t-h}{2}} I_{\frac{t-h}{2}}^{(n)}(t_1,\dots,t_n) \left(\frac{t-h}{2}-t_n\right)^{-\theta} \ud t_1\dots\ud t_n\\
\le&
\Gamma_t^n \lambda^{2n} J_{+}^2(t,x) n! 2^{-\theta} 2^n C_\theta h^\theta
\int_{0<t_1<\ldots<t_n<\frac{t-h}{2}}
J_{\frac{t-h}{2}}^{(n)}(t_1,\dots,t_n) \left(\frac{t-h}{2}-t_n\right)^{-\theta} \ud t_1\dots\ud t_n\\
=&
\Gamma_t^n \lambda^{2n} J_{+}^2(t,x) n! 2^{-\theta} 2^n C_\theta h^\theta \int_0^{\frac{t-h}{2}}\left(\frac{t-h}{2}-t_n\right)^{-\theta} k\left(\frac{2(\frac{t-h}{2}-t_n)t_n}{\frac{t-h}{2}} \right) \\
&
\times \left(\int_{0<t_1<\ldots<t_{n-1}<t_n} J_{t_n}^{(n-1)} (t_1, \ldots,t_{n-1})
 \ud t_1\dots\ud t_{n-1}\right) \ud t_n \\
\le&
\Gamma_t^n \lambda^{2n} J_{+}^2(t,x) n! 2^{-\theta} 2^n C_\theta h^\theta
2^{n-1}
\int_{0}^{\frac{t-h}{2}}h_{n-1}(t_n) \left(\frac{t-h}{2}-t_n\right)^{-\theta}
k\left(\frac{2(\frac{t-h}{2}-t_n)t_n}{\frac{t-h}{2}}\right) \ud t_n,
\end{align*}
where $I_t^{(n)}$ and $J_t^{(n)}$ are defined in Lemma 3.3 and Lemma 3.5,
and for the last inequality above we used Lemma 3.5 

We claim that for any $t>0$ and $n \geq 0$,
\begin{equation}
\label{hn-ineq2}
\int_{0}^{t}(t-s)^{-\theta} k\left( \frac{2(t-s)s}{t}\right)h_n(s) \ud s \leq 2
\int_0^t s^{-\theta} k(s) h_n(t-s)\ud s.
\end{equation}
This is proved similarly to (3.11): 
\begin{align*}
& \int_0^{t} (t-s)^{-\theta}k\left(\frac{2(t-s)s}{t}\right) h_n(s)\ud s =
\int_0^{t} s^{-\theta} k\left(\frac{2s(t-s)}{t}\right) h_n(t-s)\ud s\\
& \leq \int_0^{t/2} s^{-\theta} k\left(\frac{2s(t-s)}{t}\right) h_n(t-s)\ud s+\int_{t/2}^{t} s^{-\theta} k\left(\frac{2s(t-s)}{t}\right) h_n(s)\ud s\\
&=  \int_{0}^{t/2} s^{-\theta}k\left(\frac{2s(t-s)}{t}\right)  h_n(t-s) \ud s+
\int_{0}^{t/2} (t-s)^{-\theta} k\left(\frac{2s(t-s)}{t}\right)  h_n(t-s) \ud s\\
&\le
2\int_{0}^{t/2} s^{-\theta}k\left(\frac{2s(t-s)}{t}\right)  h_n(t-s) \ud s,
\end{align*}
where for the first inequality above we used the fact that $h_n$ is non-decreasing and hence $h_n(t-s) \leq h_n(s)$ for $s \geq t/2$, and for the last inequality we used the fact that $(t-s)^{-\theta} \leq s^{-\theta}$ for $s \in [0,t/2]$.
Because $k(t)$ is nonincreasing and $2s(t-s)/t\ge s$ for $s\in [0,t/2]$, we have that
\begin{align*}
\int_{0}^{t/2} s^{-\theta}k\left(\frac{2s(t-s)}{t}\right)  h_n(t-s) \ud s
&\le
\int_0^{t/2} s^{-\theta}k\left(s\right)h_n(t-s)\ud s\le
\int_0^{t} s^{-\theta} k\left(s \right)h_n(t-s)\ud s.
\end{align*}
This proves \eqref{hn-ineq2}.
We use inequality \eqref{hn-ineq2} with $\frac{t-h}{2}$ instead of $t$. We obtain:
\begin{align*}
 A_n'(t,h,x) \le &
\Gamma_t^n \lambda^{2n} J_{+}^2(t,x) n! 2^{-\theta} 2^{2n} C_\theta h^\theta
\int_{0}^{\frac{t-h}{2}} s^{-\theta} k(s) h_{n-1}\left(\frac{t-h}{2}-s\right)
\ud s.
\end{align*}
Using \eqref{E:sup-J0} and the fact that $h_n$ is nondecreasing, we obtain:
\begin{equation}
\label{E:Holder:An'}
 A_n'(t,h,x) \le
 \Gamma_a^n (2\lambda)^{2n} D_a n! 2^{-\theta} C_\theta h^\theta
h_{n-1}(a) \int_{0}^{a}  \frac{k(s)}{s^\theta}
\ud s.
\end{equation}

We now treat the term $B_n'(t,h,x)$. We will use \eqref{E:bound-Bn'}. Note that $1/t \geq 1/a$ and $t-s \geq t-h \geq 1/a$ for any $s \in [0,h]$. Since $k$ is non-increasing,
\begin{align*}
\int_{0}^h k\left(\frac{2s(t-s)}{t} \right)\ud s \leq & \int_{0}^h k\left(\frac{2s}{a^2} \right)\ud s \leq h^{-\theta}\int_0^h k\left(\frac{2s}{a^2} \right)\ud s=h^{\theta} \left(\frac{a^2}{2} \right)^{1-\theta}\int_0^{2h/a^2}\frac{k(s)}{s^{\theta}}\ud s.
\end{align*}
Using \eqref{E:bound-Bn'} and \eqref{E:sup-J0}, we obtain:
\begin{equation}
\label{E:Holder:Bn'}
B_n'(t,h,x) \leq \Gamma_a^n \lambda^{2n}D_a n! \, 2^{n-1}h_{n-1}(a) h^{\theta} \left(\frac{a^2}{2} \right)^{1-\theta}\int_0^{2/a}\frac{k(s)}{s^{\theta}}\ud s.
\end{equation}

Combining \eqref{moment-time}, \eqref{E:Holder:An'} and \eqref{E:Holder:Bn'}, it follows that
$$\|u(t-h,x)-u(t,x)\|_p \leq C h^{\theta/2}\widetilde{H}(a;\gamma),$$
 where $C>0$ is a constant depending on $a$ and $\beta$, and $\gamma$ is a constant depending on $p,a,\lambda$.

{\bigskip\em Step 2. (right increments in time)}
For $h>0$, we have
\begin{align}
\nonumber
\|u(t+h,x)-u(t,x)\|_p \leq & \sum_{n\geq 0}(p-1)^{n/2} \|J_n(t+h,x))-J_n(t,x)\|_2 \\
\label{moment-time-r}
\leq &\sum_{n\geq 0} (p-1)^{n/2} \left(\frac2{n!}\left[ A_n(t,h,x)+B_n(t,h,x)\right] \right)^{1/2},
\end{align}
where $A_n(t,h,x)$ and $B_n(t,h,x)$ are given by \eqref{def-A} and \eqref{def-B}, respectively.
These terms are treated similarly to $A_n'(t,h,x)$ and $B_n'(t,h,x)$ as above. We omit the details.

{\bigskip\em Step 3. (increments in space)} Let $(t,x)$ and $(t,x') \in K_a$. Say $x'=x+z$ for some $z \in \bR^d$. Then
\begin{align*}
\|u(t,x+z)-u(t,x)\|_p  \leq & \sum_{n \geq 0}(p-1)^{n/2} \|J_n(t,x+z)-J_n(t,x)\|_2\\
=& \sum_{n \geq 0}(p-1)^{n/2}\left(\frac{1}{n!}C_n(t,x,z) \right)^{1/2} \\
\leq & \sum_{n \geq 0}(p-1)^{n/2}\left(\frac{2}{n!}[C_n^{(1)}(t,x,z)+C_n^{(2)}(t,x,z)] \right)^{1/2},
\end{align*}
where $C_n(t,x,z)$, $C_n^{(1)}(t,x,z)$ and $C_n^{(2)}(t,x,z)$ are given by \eqref{def-C}, \eqref{def-C1} and \eqref{def-C2}, respectively.

Notice that for some constant $K_\theta>0$,
\[
\left|1-e^{i x}\right|^2 = 2(1-\cos(x)) \le K_\theta |x|^{2\theta},\quad\text{for all $x\in\R$.}
\]
Hence,
\begin{align*}
\left|1-\exp\left(-\frac{i}{t}\left(\sum_{j=1}^n t_j \xi_j\right)\cdot z\right)\right|^2\le&
K_{\theta} \left|\frac{(\sum_{j=1}^n t_j \xi_j) \cdot z}{t}\right|^{2\theta} \le
K_{\theta} \frac{|z|^{2\theta}}{t^{2\theta}} \left|\sum_{j=1}^n t_j \xi_j\right|^{2\theta}.
\end{align*}
From this, it follows that
\begin{align*}
C_n^{(1)}(t,x,z)  \leq &  \frac{\Gamma_t^n \lambda^{2n}J_{+}^2(t,x+z)n!}{(2\pi)^{nd}} \int_{0<t_1<\ldots<t_n<t} \int_{\bR^{nd}} \prod_{k=1}^{n}
\exp \left(- \frac{t_{k+1}-t_{k}}{t_{k}t_{k+1}}\left|\sum_{j=1}^{k}t_{j}\xi_{j}\right|^2 \right)\\
 &\times K_{\theta} \frac{|z|^{2\theta}}{t^{2\theta}} \left|\sum_{j=1}^{n}t_j \xi_j\right|^{2\theta}\mu(\ud \xi_1) \ldots \mu_n(\ud \xi_n) \ud t_1 \ldots \ud t_n.
\end{align*}
Now we bound the inner-most terms of the product using inequality \eqref{E:expo-A}:
\begin{align*}
\exp\left(-\frac{t-t_n}{t t_n}|\sum_{j=1}^{n}t_j \xi_j|^2 \right) \left|\sum_{j=1}^{n}t_j \xi_j\right|^{2\theta}  \leq & C_{\theta} \left(\frac{t-t_n}{t t_n} \right)^{-\theta} \exp\left(-\frac{t-t_n}{2t t_n}\left|\sum_{j=1}^{n}t_j \xi_j\right|^2 \right)\\
 \leq & C_{\theta} t^{2\theta} (t-t_n)^{-\theta}\exp \left(-\frac{t-t_n}{2t t_n}\left|\sum_{j=1}^{n}t_j \xi_j\right|^2 \right).
\end{align*}
Using arguments similar to those used for $A_n'(t,h,x)$ above, we obtain that
\[
C^{(1)}_n(t,x,z)\le
\Gamma_a^n (2\lambda)^{2n} D_a n! 2^{-\theta} K_{\theta}C_\theta |z|^{2\theta}
h_{n-1}(a) \int_{0}^{a}  \frac{k(s)}{s^\theta}.
\]

Finally, to treat $C_n^{(2)}(t,x,z)$, we use \eqref{def-C2}. Recalling the definition \eqref{E:def-F} of $F(t,x,z)$, the conclusion follows from the following inequality, given by Lemma 4.1 of \cite{chen-huang16}:
\begin{align}\label{E:GG-x}
 \left|G(t,x)-G(t,y)\right|\le \frac{1}{t^{\theta/2}}\left[G(2t,x)+G(2t,y)\right]|x-y|^{\theta}.
\end{align}
This completes the proof of Theorem \ref{T:Holder}.
\end{proof}
\appendix

\section{A technical lemma}
\label{A:lemmas}

\begin{lemma}\label{L:LimSupHt}
If $H:[0,\infty) \to [0,\infty)$ is a nondecreasing function such that
\[
\gamma:=\inf\left\{\beta>0:\: \int_0^\infty \e^{-\beta t} H(t)\ud t <\infty\right\}<\infty,
\]
then
\[
\limsup_{t\rightarrow\infty}\frac{1}{t}\log H(t)\le \gamma.
\]
\end{lemma}
\begin{proof}
We will prove this lemma by contradiction.
Suppose that $\limsup_{t\rightarrow\infty}t^{-1}\log H(t)>\gamma$.
Then there exist $\epsilon_0>0$ and a nondecreasing sequence $\{t_n\}_{n\ge 1}$ such that
$0\le t_n\uparrow \infty$ as $n\rightarrow \infty$ and
\[
\log H(t_n)\ge \gamma (t_n+\epsilon_0),\quad\text{for all $n\ge 1$.}
\]
Moreover, we assume that $t_n \geq s_{n-1}^*$ for a certain sequence $(s_n^*)_{n \geq 1}$ which will be constructed below.
By definition of $\gamma$, we have that
\begin{align}\label{E:gF}
\int_{0}^\infty \e^{-(\gamma+\epsilon) t}H(t)\ud t<\infty,\quad\text{for all $\epsilon>0$.}
\end{align}
We claim that there exits $s_1>t_1$ such that
$H(s_1)\le \e^{(\gamma +\epsilon_0/2)s_1}$. If this is not true, then
\[
H(t)1_{\{t> t_1\}} > \e^{(\gamma+\epsilon_0/2) t}1_{\{t>t_1\}},
\]
which leads to the following contradiction with \eqref{E:gF}:
\[
\int_{t_1}^\infty \e^{-(\gamma+\epsilon_0/4) t}H(t)\ud t
\ge
\int_{t_1}^\infty \e^{-(\gamma+\epsilon_0/4) t}\e^{(\gamma+\epsilon_0/2) t}\ud t = \infty.
\]

Let $r_1=\inf\{s_1>t_1; H(s_1) \leq e^{(\gamma+\epsilon_0/2)s_1} \}$.
Then $H(t) > e^{(\gamma+\epsilon_0/2)t}$ for any $t \in[t_1,r_1)$.
Since $H(t)$ is nondecreasing and $H(t_1) \geq e^{(\gamma+\epsilon_0)t_1}$,
the smallest possible value for $r_1$ is obtained in the case when the function $H(t)$
is constant with value equal to $\e^{(\gamma+\epsilon_0)t_1}$ starting from $t_1$ until it crosses the function $\e^{(\gamma+\epsilon_0/2)t}$. In this case, $r_1=s_1^*$ where
$\e^{(\gamma+\epsilon_0/2)s_1^*} = \e^{(\gamma+\epsilon_0)t_1}$.
For a general non-decreasing function $H$, $r_1 \geq s_1^*$.
Hence,
\[
H(t) 1_{\{t\in [t_1,s_1^*]\}} \ge \e^{(\gamma +\epsilon_0/2)t_1} 1_{\{t\in [t_1,s_1^*]\}},
\quad\text{with $s_1^*=\left(1+\frac{\epsilon_0}{2\gamma+\epsilon_0}\right)t_1$}.
\]
We now select $t_2$ such that $t_2>s_1^*$ and $t_2 \geq t_1$. In the same way, we have that
\[
H(t) 1_{\{t\in [t_2,s_2^*]\}} \ge \e^{(\gamma +\epsilon_0/2)t_2} 1_{\{t\in [t_2,s_2^*]\}},
\quad\text{with $s_2=\left(1+\frac{\epsilon_0}{2\gamma+\epsilon_0}\right)t_2$}.
\]
In this way, we can find a sequence of disjoint nonempty intervals $\{[t_n,s_n^*] \}_{n\ge 1}$ such that
\[
H(t) 1_{\{t\in [t_n,s_n^*]\}} \ge \e^{(\gamma +\epsilon_0/2)t_n} 1_{\{t\in [t_n,s_n^*]\}},
\quad\text{with $s_n^*=\left(1+\frac{\epsilon_0}{2\gamma+\epsilon_0}\right)t_n$},
\]
for all $n\ge 1$. Now we have that
\begin{align*}
\int_{0}^\infty \e^{-(\gamma+\epsilon_0/2) t}H(t)\ud t
& \ge
\sum_{n=1}^\infty
\int_{t_n}^{s_n^*} \e^{-(\gamma+\epsilon_0/2) t}\e^{(\gamma+\epsilon_0/2) t_n}\ud t\\
& =
\sum_{n=1}^\infty \frac{1}{\gamma+\epsilon_0/2}\left(1-\e^{-(\gamma+\epsilon_0/2)\frac{\epsilon_0 t_n}{2\gamma+\epsilon_0}}\right)\\
& \ge
\sum_{n=1}^\infty \frac{2}{2\gamma+\epsilon_0}\left(1-\e^{-(\gamma+\epsilon_0/2)\frac{\epsilon_0 t_1}{2\gamma+\epsilon_0}}\right)=\infty,
\end{align*}
which contradicts with \eqref{E:gF}. This proves Lemma \ref{L:LimSupHt}.
\end{proof}

\section{Continuity of $J_n$ in $L^p(\Omega)$}
\label{A:cont-Jn}

The following result is an extension of Proposition A.3 of \cite{ChenDalang13Heat} to higher dimensions $d$.

\begin{proposition}
\label{P:Btx}
Fix $(t,x)\in(0,\infty)\times\R^d$. Set
\[
B_{t,x}:=\left\{
\left(t',x'\right)\in(0,\infty) \times\R^d:\: 0< t'\le
t+\frac{1}{2}\:,\:\: \left|x'-x\right|\le 1
\right\}
\]
Then there exists $a=a_{t,x}>0$ such
that for all $\left(t',x'\right)\in B_{t,x}$ and all $s\in [0,t']$ and
$y \in \bR^d$ with $|y|\ge a$,
\begin{equation}
\label{E:Btx}
G(t'-s,x'-y) \le G(t+1-s,x-y)\:.
\end{equation}
\end{proposition}

\begin{proof}
By direct calculation, we see that inequality \eqref{E:Btx} is equivalent to
\begin{align}
\label{E:goal}
\sum_{i=1}^d \left(-\frac{(x'_i-y_i)^2}{t'-s} +\frac{(x_i-y_i)^2}{t+1-s}\right)
 \le d \: \log\left(\frac{t'-s}{t+1-s}\right),
\end{align}
where $x=(x_1,\ldots,x_d)$, $x'=(x_1',\ldots,x_d')$ and $y=(y_1,\ldots,y_d)$.

We fix $(t,x)$. In order to find $a=a_{t,x}$, we will freeze $d-1$ coordinates.
Because
\begin{align*}
-\frac{(x'_i-y_i)^2}{t'-s} &+\frac{(x_i-y_i)^2}{t+1-s}\\
&=-\frac{1+t-t'}{(1+t-s)(t'-s)}
\left(y-\frac{x'(1+t-s)-x(t'-s)}{1+t-t'}\right)^2 +\frac{(x_i-x_i')^2}{1+t-t'} \\
&\le
\frac{(x_i-x_i')^2}{1+t-t'} \le 2(x_i-x_i')^2\le 2,
\end{align*}
we have
\[
\sum_{i=1}^d \left(-\frac{(x'_i-y_i)^2}{t'-s} +\frac{(x_i-y_i)^2}{t+1-s}\right)
\le 2(d-1)
+\left(-\frac{(x'_j-y_j)^2}{t'-s} +\frac{(x_j-y_j)^2}{t+1-s}\right).
\]
for any index $j=1, \ldots,d$.
Hence, inequality \eqref{E:goal} holds, provided that there exists an index $j=1, \ldots,d$ such that
\begin{equation}
\label{one-dim}
-\frac{(x'_j-y_j)^2}{t'-s} +\frac{(x_j-y_j)^2}{t+1-s}
 \le  d \: \log\left(\frac{t'-s}{t+1-s}\right) - 2(d-1)\;.
\end{equation}
This shows that condition \eqref{E:goal} holds, if for some index $j=1,\ldots,d$, we have:
\begin{align}\label{E:subgoal1}
-\frac{(x'_j-y_j)^2}{t'-s} +\frac{(x_j-y_j)^2}{t+1-s} \le
2 d \: \log\left(\frac{t'-s}{t+1-s}\right)\;,
\end{align}
and
\begin{align}\label{E:subgoal2}
-\frac{(x'_j-y_j)^2}{t'-s} +\frac{(x_j-y_j)^2}{t+1-s}\le -4(d-1).
\end{align}

By the same argument as in the case $d=1$, there exists a constant $a_1=a_{1,t,x}>0$ such that
\eqref{E:subgoal1} and \eqref{E:subgoal2} hold for any $(t',x_j')$ with $0<t' \leq t+1/2$ and $|x_j'-x_j| \leq 1$, and for any $y_j \in \bR$ with $|y_j|>a_1$.

Let $a:=a_1 \sqrt{d}$. Note that
$\{y\in\R^d: |y|\ge  a \} \subset \bigcup_{j=1}^d B_j$, where $$B_j=\left\{y=(y_1,\ldots,y_d) \in\R^d: |y_j|\ge a_1 \right\}, \quad j=1,\ldots,d.$$
Therefore, for any $y\in\R^d$ with $|y|\ge a$,
there exists an index $j=1,\ldots,d$ such that $|y_j|\ge a_1$.
As we have shown above, this means that condition \eqref{E:goal} holds for this $y$, for any $(t',x') \in B_{t,x}$.
\end{proof}

\begin{lemma}
\label{L:J0-cont}
$J_0$ is continuous on $(0,\infty) \times \bR^d$.
\end{lemma}

\begin{proof} Fix $t>0$ and $x \in \bR^d$. By the definition of $J_0$, we have:
$$|J_0(t,x)-J_0(t',x')| \leq  \int_{\bR^d}|G(t,x-y)-G(t',x'-y)|\:|u_0|(\ud y)=:L(t,t',x,x').$$

We claim that:
\begin{equation}
\label{E:L-conv}
\lim_{(t',x') \to (t,x)}L(t,t',x,x')=0.
\end{equation}

\noindent To see this, we write $L(t,t',x,x')=L_1(t,t',x,x')+L_2(t,t',x,x')$ where
\begin{align*}
L_1(t,t',x,x')&= \int_{|y| \geq a}|G(t,x-y)-G(t',x'-y)|\:|u_0|(\ud y), \quad\text{and}\\
L_2(t,t',x,x')&= \int_{|y|<a}|G(t,x-y)-G(t',x'-y)|\:|u_0|(\ud y),
\end{align*}
and $a=a_{t,x}$ is the constant given by Proposition \ref{P:Btx}. By enlarging $a$ if necessary, we may assume that $t>1/a$. By the dominated convergence theorem and the continuity of the function $G$, we see that $L_i(t,t',x,x') \to 0$ when $(t',x') \to (t,x)$, for $i=1,2$. To justify the application of this theorem, we argue as follows. For $L_1(t,t',x,x')$, we use Proposition \ref{P:Btx} to infer that for any $(t',x') \in B_{t,x}$ and for any $y \in \bR^d$ with $|y| \geq a$, we have:
$$|G(t,x-y)-G(t',x'-y)| \leq 2 G(t+1,x-y).$$
For $L_2(t,t',x,x')$, we use the fact that for any $t'>1/a$, $x' \in \bR^d$ and $y\in \bR^d$ with $|y| \leq a$,
$$\frac{G(t',x'-y)}{G(t,x-y)}=\frac{\sqrt{t}}{\sqrt{t'}}\exp \left(-\frac{(x'-y)^2}{2t'}+\frac{(x-y)^2}{2t} \right) \leq \frac{\sqrt{t}}{\sqrt{1/a}}\exp\left( \frac{|x|^2+|a|^2}{t}\right)=:C_{t,x},$$
and hence $|G(t',x'-y)-G(t,x-y)| \leq (C_{t,x}+1)G(t,x-y)$.
\end{proof}

\begin{lemma}
\label{L:Jn-cont}
For any $p \geq 2$ and $n \geq 1$, $J_n$ is $L^p(\Omega)$-continuous on $(0,\infty) \times \bR^d$.
\end{lemma}

\begin{proof} We proceed as in the proof of Lemma 3.6 of \cite{balan-song16}. We divide the proof in three steps.

{\em Step 1. (right-continuity in time)}  We will prove that for any $t>0$ and $a>0$,
\begin{equation}
\label{E:t-cont1}
\lim_{h \downarrow 0}\|J_n(t+h,x)-J_n(t,x)\|_p=0 \quad \mbox{uniformly in} \ x \in [-a,a]^d.
\end{equation}
For any $h>0$, we have:
\begin{eqnarray}
\nonumber
\|J_n(t+h,x)-J_n(t,x)\|_p^2 & \leq &  (p-1)^n \|J_n(t+h,x)-J_n(t,x)\|_2^2 \\
\nonumber
& = & (p-1)^n n! \,
\|\widetilde{f}_n(\cdot,t+h,x)-\widetilde{f}_n(\cdot,t,x)\|_{\cH^{\otimes n}}^2 \\
\label{E:AB}
& \leq & \frac{2}{n!} \left(A_n(t,x,h)+B_n(t,x,h) \right),
\end{eqnarray}
where
\begin{align}
\label{def-A}
A_n(t,x,h)&= (n!)^2 \|\widetilde{f}_n(\cdot,t+h,x)1_{[0,t]^{n}}-\widetilde{f}_n(\cdot,t,x) \|_{\cH^{\otimes n}}^{2},\\
\label{def-B}
B_n(t,x,h)&= (n!)^2 \|\widetilde{f}_n(\cdot,t+h,x)1_{[0,t+h]^{n}\setminus [0,t]^n} \|_{\cH^{\otimes n}}^{2}.
\end{align}

We evaluate $A_n(t,h,x)$ first. We have:
$$A_n(t,h,x) = \int_{[0,t]^{2n}}
\prod_{j=1}^{n}\gamma(t_j-s_j) \psi_{t,h,x}^{(n)}({\bf t}, {\bf
s})\ud{\bf t} \ud{\bf s},$$
where
\begin{align*}
\psi_{t,h,x}^{(n)}({\bf t},{\bf s}) =\frac{1}{(2\pi)^{nd}}\int_{\bR^{nd}}
 &\cF (g_{{\bf t},t+h,x}^{(n)}-g_{{\bf t},t,x}^{(n)}) (\xi_1,\ldots,\xi_n) \\
\times& \overline{\cF (g_{{\bf s},t+h,x}^{(n)}-g_{{\bf s},t,x}^{(n)}) (\xi_1, \ldots,\xi_n)}
\:\: \mu(\ud\xi_1)
\ldots \mu(\ud\xi_n).
\end{align*}
Similarly to \eqref{estimate-alpha1}, we have:
\begin{equation}
\label{E:A}
A_n(t,h,x) \leq \Gamma_t^n \int_{[0,t]^n}\psi_{t,h,x}^{(n)}({\bf t},{\bf t})\ud {\bf  t}=\Gamma_t^n \sum_{\rho \in S_n}\int_{0<t_{\rho(1)}<\ldots<t_{\rho(n)}<t}\psi_{t,h,x}^{(n)}({\bf t},{\bf t})\ud {\bf  t}.
\end{equation}
If $t_{\rho(1)}<\ldots<t_{\rho(n)}<t=:t_{\rho(n+1)}$, then by \eqref{E:Fourier-gt},
\begin{eqnarray*}
\lefteqn{|\cF (g_{{\bf t},t+h,x}^{(n)}-g_{{\bf t},t,x}^{(n)}) (\xi_1,\ldots,\xi_n) |^2 \leq  \lambda^{2n}
J_+^2(t,x) \prod_{k=1}^{n-1} \exp
\left( -\frac{t_{\rho(k+1)}-t_{\rho(k)}}{t_{\rho(k)} t_{\rho(k+1)}} \left|\sum_{j=1}^{k}t_{\rho(j)} \xi_{\rho(j)} \right|^2 \right) } \\
& & \left| \exp \left( -\frac{1}{2}\frac{t+h-t_{\rho(n)}}{t_{\rho(n)}(t+h)} \left|\sum_{j=1}^{n}t_{j} \xi_{j} \right|^2 \right)-\exp\left(
-\frac{1}{2}\frac{t-t_{\rho(n)}}{t_{\rho(n)}t} \left|\sum_{j=1}^{n}t_{j} \xi_{j} \right|^2 \right)
\right|^2 \\
& & = \lambda^{2n} J_+^2(t,x) \prod_{k=1}^{n} \exp
\left( -\frac{t_{\rho(k+1)}-t_{\rho(k)}}{t_{\rho(k)} t_{\rho(k+1)}} \left|\sum_{j=1}^{k}t_{\rho(j)} \xi_{\rho(j)} \right|^2 \right)
\left[ 1-\exp \left(-\frac{h}{2t(t+h)}\left|\sum_{j=1}^{n}t_j \xi_j\right|^2 \right) \right]^2,
\end{eqnarray*}
and hence
\begin{align}
\nonumber
\psi_{t,h,x}^{(n)}({\bf t},{\bf t}) \leq & \Gamma_t^n J_+^2(t,x) \frac{1}{(2\pi)^{nd}}\int_{\bR^{nd}}
\prod_{k=1}^{n} \exp
\left( -\frac{t_{\rho(k+1)}-t_{\rho(k)}}{t_{\rho(k)} t_{\rho(k+1)}} \left|\sum_{j=1}^{k}t_{\rho(j)} \xi_{\rho(j)} \right|^2 \right)\\
\label{E:psi-h}
 &\times \left[ 1-\exp \left(-\frac{h}{2t(t+h)}\left|\sum_{j=1}^{n}t_j \xi_j\right|^2 \right) \right]^2 \mu(\ud \xi_1)\ldots \mu_n(\ud \xi_n).
\end{align}
Using \eqref{E:A} and \eqref{E:psi-h}, it follows that
\begin{align}
\nonumber
A_n(t,h,x) \leq & \Gamma_t^n \lambda^{2n}J_+^2(t,x) n! \frac{1}{(2\pi)^{nd}} \int_{0<t_1<\ldots<t_n<t} \int_{\bR^{nd}} \prod_{k=1}^{n}\exp \left(-\frac{t_{k+1}-t_k}{t_k t_{k+1}} \left|\sum_{j=1}^{k}t_j \xi_j \right|^2\right) \\
\label{E:bound-An}
 & \times \left[ 1-\exp \left(-\frac{h}{2t(t+h)}\left|\sum_{j=1}^{n}t_j \xi_j\right|^2 \right) \right]^2 \mu(\ud \xi_1) \ldots \mu(\ud \xi_n) \ud t_1 \ldots \ud t_n,
\end{align}
with the convention $t_{n+1}=t$.
By the dominated convergence theorem and \eqref{E:sup-J0}, we conclude that
\begin{equation}
\label{E:A-conv}
\lim_{h \downarrow 0}A_n(t,h,x)= 0 \quad \mbox{uniformly in $x \in [-a,a]^d$}.
\end{equation}

As for $B_n(t,h,x)$, note that
$$B_n(t,h,x) =\int_{[0,t+h]^{2n}}
\prod_{j=1}^{n}\gamma(t_j-s_j)
\gamma_{t,h,x}^{(n)}({\bf t}, {\bf s})1_{D_{t,h}}({\bf t}) 1_{D_{t,h}}({\bf s})\ud{\bf t} d{\bf s},
$$
where $D_{t,h}=[0,t+h]^n \verb2\2 [0,t]^n$ and
$$\gamma_{t,h,x}^{(n)}({\bf t},{\bf s}) = \frac{1}{(2\pi)^{nd}}\int_{\bR^{nd}} \cF g_{{\bf
t},t+h,x}^{(n)}(\xi_1, \ldots,\xi_n)\overline{\cF g_{{\bf s},t+h,x}^{(n)}(\xi_1, \ldots,\xi_n)} \mu(\ud \xi_1) \ldots \mu(\ud \xi_n).$$
Similarly to (46) of \cite{balan-song16}, it can be proved that
\begin{equation}
\label{E:B}
B_n(t,h,x) \leq \Gamma_{t+h}^n \int_{[0,t+h]^n} \gamma_{t,h,x}^{(n)}({\bf t},{\bf t})1_{D_{t,h}}({\bf t})\ud {\bf t}.
\end{equation}
If $t_{\rho(1)}<\ldots<t_{\rho(n)}<t+h$, then by \eqref{E:Fourier-gt},
\begin{align*}
|\cF g_{{\bf t},t+h,x}(\xi_1,\ldots,\xi_n)|^2  \leq & \lambda^{2n} J_+^2(t,x)
\prod_{k=1}^{n-1}\exp \left(-\frac{t_{\rho(k+1)}-t_{\rho(k)}}{t_{\rho(k)}t_{\rho(k+1)}}  \left|\sum_{j=1}^k t_{\rho(j)}\xi_{\rho(j)} \right|^2\right) \\
 &\times \exp \left(-\frac{t+h-t_{\rho(n)}}{(t+h)t_{\rho(n)}}  \left|\sum_{j=1}^n t_{j}\xi_{j} \right|^2\right),
\end{align*}
and hence, by Lemma \ref{max-principle-lemma}
\begin{align}
\nonumber
\gamma_{t,h,x}^{(n)}({\bf t},{\bf t})  \leq & \lambda^{2n}J_+^2(t,x) \frac{1}{(2\pi)^{nd}}\int_{\bR^{nd}}\prod_{k=1}^{n-1}\exp \left(-\frac{t_{\rho(k+1)}-t_{\rho(k)}}{t_{\rho(k)}t_{\rho(k+1)}}  \left|\sum_{j=1}^k t_{\rho(j)}\xi_{\rho(j)} \right|^2\right) \\
\nonumber
 &\times \exp \left(-\frac{t+h-t_{\rho(n)}}{(t+h)t_{\rho(n)}}  \left|\sum_{j=1}^n t_{j}\xi_{j} \right|^2\right) \mu(\ud \xi_1) \ldots \mu_n(\ud \xi_n) \\
\nonumber
 \leq & \lambda^{2n} J_+^2(t,x) \frac{1}{(2\pi)^{nd}}  \prod_{k=1}^{n-1} \int_{\bR^d}
\exp \left(-\frac{t_{\rho(k+1)}-t_{\rho(k)}}{t_{\rho(k)}t_{\rho(k+1)}}  \left|t_{\rho(k)}\xi_{k} \right|^2\right)\mu(\ud \xi_k) \\
\label{E:gamma-h}
 & \times \int_{\bR^d}\exp \left(-\frac{t+h-t_{\rho(n)}}{(t+h)t_{\rho(n)}}  \left|t_{\rho(n)}\xi_{n} \right|^2\right)\mu(\ud \xi_n).
\end{align}

Using relations \eqref{E:B} and \eqref{E:gamma-h},
and the fact that $$D_{t,h}=\bigcup_{\rho \in S_n}\{(t_1,\ldots,t_n);0<t_{\rho(1)}<\ldots<t_{\rho(n)}<t+h, \, t_{\rho(n)}>t\},$$  we obtain that
\begin{align}
\nonumber
B_n(t,h,x)  \leq & \Gamma_{t+h}^n \sum_{\rho \in S_n} \int_{t}^{t+h} \int_{0<t_{\rho(1)}<\ldots<t_{\rho(n-1)}<t_{\rho(n)}} \gamma_{t,h,x}^{(n)}({\bf t},{\bf t})\ud t_{\rho(1)} \ldots
\ud t_{\rho(n-1)} \ud t_{\rho(n)} \\
\nonumber
 \leq & \Gamma_{t+h}^n \lambda^{2n} J_+^2(t,x)  n! \, \frac{1}{(2\pi)^{nd}} \int_{t}^{t+h} \int_{0<t_{1}<\ldots<t_{n-1}<t_{n}} \prod_{k=1}^{n-1} \int_{\bR^d}
\exp \left(-\frac{t_{k+1}-t_{k}}{t_{k}t_{k+1}}  \left|t_{k}\xi_{k} \right|^2\right)\mu(\ud \xi_k)\\
\nonumber
 & \times \int_{\bR^d}\exp \left(-\frac{t+h-t_{n}}{(t+h)t_{n}}  \left|t_{n}\xi_{n} \right|^2\right) \mu(\ud \xi_n)
\ud t_{1} \ldots \ud t_{n-1} \ud t_{n} \\
\nonumber
=& \Gamma_{t+h}^n \lambda^{2n} J_+^2(t,x) n! \, \int_{t}^{t+h} \int_{0<t_{1}<\ldots<t_{n-1}<t_{n}} J_{t_n}^{(n-1)}(t_1,\ldots,t_{n-1}) k\left(\frac{2(t+h-t_n)t_n}{t+h}\right)\ud t_n \\
\nonumber
 \leq & \Gamma_{t+h}^n \lambda^{2n} J_+^2(t,x) n! \, 2^{n-1}\int_{t}^{t+h}h_{n-1}(t_n)
k\left(\frac{2(t+h-t_n)t_n}{t+h}\right)\ud t_n \\
\nonumber
=& \Gamma_{t+h}^n \lambda^{2n}J_+^2(t,x) n! \, 2^{n-1}\int_{0}^{h}h_{n-1}(t+s)
k\left(\frac{2(h-s)(t+s)}{t+h}\right) \ud s \\
\label{E:bound-Bn}
\leq & \Gamma_{t+h}^n \lambda^{2n} J_+^2(t,x) n! \, 2^{n-1} h_{n-1}(t+h) \int_0^h k \left(\frac{2(h-s)t}{t+h} \right)\ud s
\end{align}
where the second last inequality is due to Lemma \ref{L:Int-I}, and for the last inequality we used the fact that $h_{n-1}$ is non-decreasing and $k$ is non-increasing.
By the dominated convergence theorem and \eqref{E:sup-J0}, we infer that
\begin{equation}
\label{E:B-conv}
\lim_{h \downarrow 0}B_n(t,h,x) = 0 \quad \mbox{uniformly in $x \in [-a,a]^d$}.
\end{equation}

Relation \eqref{E:t-cont1} follows from \eqref{E:AB}, \eqref{E:A-conv} and \eqref{E:B-conv}.

\vspace{3mm}

{\em Step 2. (left-continuity in time)} We will prove that for any $t>0$ and $a>0$,
\begin{equation}
\label{E:t-cont2}
\lim_{h \downarrow 0}\|J_n(t-h,x)-J_n(t,x)\|_p=0 \quad \mbox{uniformly in} \ x \in [-a,a]^d.
\end{equation}
For any $h>0$, we have:
\begin{eqnarray}
\nonumber
\|J_n(t-h,x)-J_n(t,x)\|_p^2 & \leq &  (p-1)^n \|J_n(t-h,x)-J_n(t,x)\|_2^2 \\
\nonumber
& = & (p-1)^n n! \,
\|\widetilde{f}_n(\cdot,t-h,x)-\widetilde{f}_n(\cdot,t,x)\|_{\cH^{\otimes n}}^2 \\
\label{E:AB'}
& \leq & \frac{2}{n!} \left(A_n'(t,x,h)+B_n'(t,x,h) \right),
\end{eqnarray}
where
\begin{eqnarray}
\label{def-A'}
A_n'(t,x,h)&=& (n!)^2 \|\widetilde{f}_n(\cdot,t-h,x)-\widetilde{f}_n(\cdot,t,x) 1_{[0,t-h]^n}\|_{\cH^{\otimes n}}^{2},\\
\label{def-B'}
B_n'(t,x,h)&=& (n!)^2 \|\widetilde{f}_n(\cdot,t,x)1_{[0,t]^{n}\verb2\2 [0,t-h]^n} \|_{\cH^{\otimes n}}^{2}.
\end{eqnarray}

We evaluate $A_n'(t,h,x)$ first. We have:
$$A_n'(t,h,x) = \int_{[0,t-h]^{2n}}
\prod_{j=1}^{n}\gamma(t_j-s_j) \psi_{t,h,x}^{(n)'}({\bf t}, {\bf
s})\ud{\bf t} \ud{\bf s},$$
where
\begin{align*}
\psi_{t,h,x}^{(n)}({\bf t},{\bf s})' =&\frac{1}{(2\pi)^{nd}}\int_{\bR^{nd}}
\cF (g_{{\bf t},t,x}^{(n)}-g_{{\bf t},t-h,x}^{(n)}) (\xi_1,\ldots,\xi_n) \\
 &\times \overline{\cF (g_{{\bf s},t,x}^{(n)}-g_{{\bf s},t-h,x}^{(n)}) (\xi_1, \ldots,\xi_n)}\mu(\ud\xi_1)
\ldots \mu(\ud\xi_n).
\end{align*}
Similarly to \eqref{estimate-alpha1}, we have:
\begin{equation}
\label{E:A'}
A_n'(t,h,x) \leq \Gamma_{t-h}^n \int_{[0,t-h]^n}\psi_{t,h,x}^{(n)'}({\bf t},{\bf t})\ud {\bf  t}=\Gamma_{t-h}^n \sum_{\rho \in S_n}\int_{0<t_{\rho(1)}<\ldots<t_{\rho(n)}<t-h}\psi_{t,h,x}^{(n)'}({\bf t},{\bf t})\ud {\bf  t}.
\end{equation}
If $t_{\rho(1)}<\ldots<t_{\rho(n)}<t-h$, then by \eqref{E:Fourier-gt},
\begin{align*}
&\hspace{-2em}|\cF (g_{{\bf t},t,x}^{(n)}-g_{{\bf t},t-h,x}^{(n)}) (\xi_1,\ldots,\xi_n) |^2 \\
\leq&  \lambda^{2n} J_+^2(t,x) \prod_{k=1}^{n-1}
\exp\left( -\frac{t_{\rho(k+1)}-t_{\rho(k)}}{t_{\rho(k)} t_{\rho(k+1)}} \left|\sum_{j=1}^{k}t_{\rho(j)} \xi_{\rho(j)} \right|^2 \right)  \\
 & \times\left| \exp \left( -\frac{1}{2}\frac{t-t_{\rho(n)}}{tt_{\rho(n)}} \left|\sum_{j=1}^{n}t_{j} \xi_{j} \right|^2 \right)-\exp\left(
-\frac{1}{2}\frac{t-h-t_{\rho(n)}}{(t-h)t_{\rho(n)}} \left|\sum_{j=1}^{n}t_{j} \xi_{j} \right|^2 \right)
\right|^2 \\
=&  \lambda^{2n} J_+^2(t,x) \prod_{k=1}^{n} \exp
\left( -\frac{t_{\rho(k+1)}-t_{\rho(k)}}{t_{\rho(k)} t_{\rho(k+1)}} \left|\sum_{j=1}^{k}t_{\rho(j)} \xi_{\rho(j)} \right|^2 \right)\\
&\times \exp \left(-\frac{t-h-t_{\rho(n)}}{(t-h)t_{\rho(n)}}\left|\sum_{j=1}^n t_j \xi_j  \right|^2 \right)  \left[ 1-\exp \left(-\frac{h}{2t(t-h)}\left|\sum_{j=1}^{n}t_j \xi_j\right|^2 \right) \right]^2,
\end{align*}
and hence
\begin{align}
\nonumber
\psi_{t,h,x}^{(n)'}({\bf t},{\bf t}) \leq & \Gamma_t^n J_+^2(t,x) \frac{1}{(2\pi)^{nd}}\int_{\bR^{nd}}
\prod_{k=1}^{n-1} \exp
\left( -\frac{t_{\rho(k+1)}-t_{\rho(k)}}{t_{\rho(k)} t_{\rho(k+1)}} \left|\sum_{j=1}^{k}t_{\rho(j)} \xi_{\rho(j)} \right|^2 \right) \\
\nonumber
 & \times\exp \left(-\frac{t-h-t_{\rho(n)}}{(t-h)t_{\rho(n)}}\left|\sum_{j=1}^n t_j \xi_j  \right|^2 \right) \\
\label{E:psi-h'}
 &
\times \left[ 1-\exp \left(-\frac{h}{2t(t-h)}\left|\sum_{j=1}^{n}t_j \xi_j\right|^2 \right) \right]^2 \mu(\ud \xi_1)\ldots \mu_n(\ud \xi_n).
\end{align}

It follows that
\begin{align}
\nonumber
A_n'(t,h,x)  \leq & \Gamma_t^n \lambda^{2n}J_+^2(t,x) n! \frac{1}{(2\pi)^{nd}} \int_{0<t_1<\ldots<t_n<t-h} \int_{\bR^{nd}} \prod_{k=1}^{n-1}\exp \left(-\frac{t_{k+1}-t_k}{t_k t_{k+1}} \left|\sum_{j=1}^{k}t_j \xi_j \right|^2\right) \\
\nonumber
 & \times \exp \left(-\frac{t-h-t_{n}}{(t-h)t_{n}}\left|\sum_{j=1}^n t_j \xi_j  \right|^2 \right)\\
\label{E:bound-An'}
 & \times \left[ 1-\exp \left(-\frac{h}{2t(t-h)}\left|\sum_{j=1}^{n}t_j \xi_j\right|^2 \right) \right]^2 \mu(\ud \xi_1) \ldots \mu(\ud \xi_n) \ud t_1 \ldots \ud t_n,
\end{align}
We will now prove that
\begin{equation}
\label{E:A'-conv}
\lim_{h \downarrow 0}A_n'(t,h,x)= 0 \quad \mbox{uniformly in $x \in [-a,a]^d$}.
\end{equation}
For this, we assume that $h\in [0,t/2]$.
Notice that:
\begin{multline*}
 \exp\left(-\frac{t-h-t_n}{(t-h)t_n}\left|\sum_{j=1}^n t_j \xi_j\right|^2\right)
\left[1-\exp\left(-\frac{h}{2t(t-h)}\left|\sum_{j=1}^n t_j \xi_j\right|^2\right)\right]^{2}\\
\le
\exp\left(-\frac{t-h-t_n}{(t-h)t_n}\left|\sum_{j=1}^n t_j \xi_j\right|^2\right)
\min\left(\frac{h}{t^{2}}\left|\sum_{j=1}^n t_j \xi_j\right|^{2},1\right).
\end{multline*}
For this, we used the fact that $(1-e^{-x})^2 \leq 1-e^{-x} \leq \min(x,1)$ for $x>0$.
Now we move the exponential inside $\min(...)$ and consider the two competing terms separately.
For $A>0$ and $x\ge 0$, we see that
\begin{align}\label{E:OptA}
\exp\left(-\frac{A}{2} x^2\right) x^{2}
=\exp\left(-\frac{A}{2} x^2 + 2 \log x\right)
\le (2/e) A^{-1}.
\end{align}
This can be seen by noticing that the function $f(x)=-\frac{A}{2} x^2 +2 \log x,x>0$ attains its maximum at $x_0=\sqrt{2/A}$.
Hence, inequality \eqref{E:OptA} implies that
\begin{align*}
\exp\left(-\frac{t-h-t_n}{(t-h)t_n}\left|\sum_{j=1}^n t_j \xi_j\right|^2\right)
\frac{h}{t^2}\left|\sum_{j=1}^n t_j \xi_j\right|^{2}
& \le \frac{2h}{e t^2}
\frac{(t-h)t_n}{t-h-t_n}
\exp\left(-\frac{t-h-t_n}{2(t-h)t_n}\left|\sum_{j=1}^n t_j \xi_j\right|^2\right)\\
& \le
\frac{2h}{e(t-h-t_n)}
\exp\left(-\frac{t-h-t_n}{2(t-h)t_n}\left|\sum_{j=1}^n t_j \xi_j\right|^2\right).
\end{align*}
The second term is bounded by
\[
\exp\left(-\frac{t-h-t_n}{(t-h)t_n}\left|\sum_{j=1}^n t_j \xi_j\right|^2\right)
\le
\exp\left(-\frac{t-h-t_n}{2(t-h)t_n}\left|\sum_{j=1}^n t_j \xi_j\right|^2\right).
\]
Therefore,
\begin{align*}
\exp&\left(-\frac{t-h-t_n}{(t-h)t_n}\left|\sum_{j=1}^n t_j \xi_j\right|^2\right)
\min\left(\frac{h}{t^{2}}\left|\sum_{j=1}^n t_j \xi_j\right|^{2},1\right)\\
&\le
\exp\left(-\frac{t-h-t_n}{2(t-h)t_n}\left|\sum_{j=1}^n t_j \xi_j\right|^2\right)
\min\left(\frac{2h}{e(t-h-t_n)},1\right).
\end{align*}
Putting the above bounds back into the expression of $A_n'(t,h,x)$, we see that
\begin{align*}
 A_n'(t,h,x) \le  & \Gamma_t^n \lambda^{2n} J_+^2(t,x) n!
 \frac{1}{(2\pi)^{nd}}
\int_{0<t_1<\ldots<t_n<t-h} \ud t_1 \ldots \ud t_n \int_{\bR^{nd}} \mu(\ud \xi_1) \ldots \mu(\ud \xi_n)\\
&\times \prod_{k=1}^{n-1}\exp \left(-\frac{t_{k+1}-t_k}{2t_k t_{k+1}} \left|\sum_{j=1}^{k}t_j \xi_j \right|^2\right)\\
&\times \exp \left(-\frac{t-h-t_n}{2(t-h)t_n} \left|\sum_{j=1}^{n}t_j \xi_j \right|^2\right)
\min\left(\frac{2h}{e(t-h-t_n)},1\right)\\
=:&\Gamma_t^n \lambda^{2n} J_+^2(t,x) n!\: A''_n(t,h).
\end{align*}
Relation \eqref{E:A'-conv} will follow from \eqref{E:sup-J0}, once we prove that:
\begin{align}\label{E_:InZero}
 \lim_{h\rightarrow 0} A_n''(t,h)  =0 .
\end{align}

We will use the fact that
\begin{equation}
\label{E:half-tn}
\frac{1}{2}\frac{t_{k+1}-t_k}{t_{k}t_{k+1}} \left|\sum_{j=1}^{k}t_j \xi_j\right|^2=
\frac{\frac{t_{k+1}}{2}-\frac{t_k}{2}}{\frac{t_{k+1}}{2}\frac{t_k}{2}}
\left|\sum_{j=1}^{k}\frac{t_j}{2}\xi_j\right|^2
\end{equation}
for any $k=1, \ldots,n$, with $t_{n+1}=t-h$.
Using the change of variables $t_k'=t_k/2$ for $k=1,\ldots,n$, and recalling the definition of the integral $I_{t}^{(n)}(t_1,\ldots,t_n)$ given in Lemma \ref{L:Jn}, we see that
\begin{align*}
A_n''(t,h) & =  2^n \int_{0<t_1<\ldots<t_{n}<\frac{t-h}{2}}\min\left(\frac{h}{e(\frac{t-h}{2}-t_n)},1 \right)I_{\frac{t-h}{2}}^{(n)}(t_1,\ldots,t_n)\ud t_1 \ldots \ud t_n\\
& \leq  2^n \int_{0<t_1<\ldots<t_{n}<\frac{t-h}{2}}\min\left(\frac{h}{e(\frac{t-h}{2}-t_n)},1 \right)J_{\frac{t-h}{2}}^{(n)}(t_1,\ldots,t_n)\ud t_1 \ldots \ud t_n \\
&= 2^n \int_{0}^{\frac{t-h}{2}} \min\left(\frac{h}{e(\frac{t-h}{2}-t_n)},1 \right)
k\left(\frac{2(\frac{t-h}{2}-t_n)t_n}{\frac{t-h}{2}}\right)\\
& \quad\times\left(\int_{0<t_1<\ldots<t_{n-1}<t_{n}}J_{t_n}^{(n-1)}(t_1,\ldots,t_{n-1})\ud t_1 \ldots \ud t_{n-1}\right) \ud t_n\\
& \leq 2^{2n-1} \int_{0}^{\frac{t-h}{2}} \min\left(\frac{h}{e(\frac{t-h}{2}-s)},1 \right)
k\left(\frac{2(\frac{t-h}{2}-s)s}{\frac{t-h}{2}}\right) h_{n-1}(s) \ud s\\
& \leq 2^{2n} \int_{0}^{\frac{t-h}{2}} \min\left(\frac{h}{es},1 \right)k(s)h_{n-1}(t-s) \ud s,
\end{align*}
where the first inequality is due to Lemma \ref{L:I-bnd}, the second last inequality is due to Lemma \ref{L:Int-I}, and the last inequality can be proved similarly to \eqref{hn-ineq}. By the dominated convergence theorem, the last integral converges to $0$ as $h \to 0$, because $\int_0^{t} k(s) h_{n-1}(t-s) \ud s = h_n(t)<\infty$. This concludes the proof of \eqref{E_:InZero}.

As for $B_n'(t,h,x)$, note that
$$B_n'(t,h,x) =\int_{[0,t]^{2n}}
\prod_{j=1}^{n}\gamma(t_j-s_j)
\psi_{t,x}^{(n)}({\bf t}, {\bf s})1_{D_{t,h}'}({\bf t}) 1_{D_{t,h}'}({\bf s})\ud{\bf t} \ud{\bf s},
$$
where $D_{t,h}'=[0,t]^n \verb2\2 [0,t-h]^n$ and $\psi_{t,x}^{(n)}({\bf t},{\bf s})$ is given by \eqref{E:def-psi}.

Similarly to \eqref{E:B}, we have:
\begin{equation}
\label{E:B'}
B_n'(t,h,x) \leq \Gamma_{t}^n \int_{[0,t]^n} \psi_{t,x}^{(n)}({\bf t},{\bf t})1_{D_{t,h}'}({\bf t})\ud {\bf t}.
\end{equation}

Using Lemmas \ref{le-lemma}, \ref{max-principle-lemma} and \ref{L:Int-I}, and the fact that
$$D_{t,h}'=\bigcup_{\rho \in S_n}\{(t_1,\ldots,t_n);0<t_{\rho(1)}<\ldots<t_{\rho(n)}<t, \, t_{\rho(n)}>t-h\},$$  we obtain that
\begin{align}
\nonumber
B_n'(t,h,x)  \leq & \Gamma_{t}^n \sum_{\rho \in S_n} \int_{t-h}^{t} \int_{0<t_{\rho(1)}<\ldots<t_{\rho(n-1)}<t_{\rho(n)}} \psi_{t,x}^{(n)}({\bf t},{\bf t})\ud t_{\rho(1)} \ldots
\ud t_{\rho(n-1)} \ud t_{\rho(n)} \\
\nonumber
 \leq & \Gamma_{t}^n \lambda^{2n} J_+^2(t,x) n! \, \frac{1}{(2\pi)^{(n-1)d}}
 \int_{t-h}^{t} \int_{0<t_{1}<\ldots<t_{n-1}<t_{n}} \ud t_{1} \ldots \ud t_{n-1} \ud t_{n}\\
 \nonumber
&\times\prod_{k=1}^{n-1} \int_{\bR^d}
\exp \left(-\frac{t_{k+1}-t_{k}}{t_{k}t_{k+1}}  \left|t_{k}\xi_{k} \right|^2\right)\mu(\ud \xi_k)\\
\nonumber
 & \times \left(\frac{1}{(2\pi)^d}\int_{\bR^d} \exp \left(-\frac{t-t_{n}}{t_nt}  \left|t_n\xi_{n} \right|^2\right)\mu(\ud \xi_n) \right) \\
\nonumber
=& \Gamma_{t}^n \lambda^{2n} J_+^2(t,x) n! \, \int_{t-h}^{t} \ud t_n \: k\left( \frac{2(t-t_n)t_n}{t}\right)\\
\nonumber
&\times \int_{0<t_{1}<\ldots<t_{n-1}<t_{n}} J_{t_n}^{(n-1)}(t_1,\ldots,t_{n-1})\ud t_1 \ldots \ud t_{n-1} \\
\nonumber
 \leq & \Gamma_{t}^n \lambda^{2n} J_+^2(t,x) n! \, 2^{n-1}\int_{t-h}^{t}h_{n-1}(t_n)
k\left(\frac{2(t-t_n)t_n}{t}\right)\ud t_n \\
\label{E:bound-Bn'}
 \leq & \Gamma_{t}^n \lambda^{2n} J_+^2(t,x) n! \, 2^{n-1}h_{n-1}(t) \int_{0}^{h}
k\left(\frac{2s(t-s)}{t}\right)\ud s.
\end{align}
By the dominated convergence theorem and \eqref{E:sup-J0}, it follows that
\begin{equation}
\label{E:B'-conv}
\lim_{h \downarrow 0}B_n'(t,h,x)= 0 \quad \mbox{uniformly in $x \in [-a,a]^d$}.
\end{equation}
Relation \eqref{E:t-cont2} follows from \eqref{E:AB'}, \eqref{E:A'-conv} and \eqref{E:B'-conv}.

\vspace{3mm}

{\em Step 3. (continuity in space)} We will prove that for any $t>0$ and $x \in \bR^d$,
\begin{equation}
\label{E:x-cont}
\lim_{|z| \to 0}\|J_n(t,x+z)-J_n(t,x)\|_p=0.
\end{equation}
For any $z \in \bR^d$, we have
\begin{eqnarray}
\nonumber
\|J_n(t,x+z)-J_n(t,x)\|_p & \leq & (p-1)^{n} \|J_n(t,x+z)-J_n(t,x)\|_2^2 \\
\label{E:C}
&=& (p-1)^n \frac{1}{n!} \, C_n(t,x,z),
\end{eqnarray}
where
\begin{eqnarray}
\label{def-C}
C_n(t,x,z)&=& (n!)^2 \|\widetilde{f}_n(\cdot,t,x+z)-\widetilde{f}_n(\cdot,t,x)\|_{\cH^{\otimes n}}^2 \\
\nonumber
&=& \int_{[0,t]^{2n}} \prod_{j=1}^{n}\gamma(t_j-s_j) \psi_{t,x,z}^{(n)}({\bf t},{\bf s})\ud{\bf t}\ud{\bf s}
\end{eqnarray}
and
\begin{align*}
\psi_{t,x,z}^{(n)}({\bf t},{\bf s})=& \frac{1}{(2\pi)^{nd}}  \int_{\bR^d} \cF (g_{{\bf t},t,x+z}^{(n)}-g_{{\bf t},t,x}^{(n)})(\xi_1,\ldots,\xi_n)\\
 & \times \overline{\cF (g_{{\bf s},t,x+z}^{(n)}-g_{{\bf s},t,x}^{(n)})(\xi_1,\ldots,\xi_n)} \mu(\ud \xi_1) \ldots \mu(\ud \xi_n).
\end{align*}
Similarly to the previous estimates, we have:
\begin{equation}
\label{E:bound-Cn}
C_n(t,x,z)\leq  \Gamma_t^n \int_{[0,t]^n}\psi_{t,x,z}^{(n)}({\bf t},{\bf t})\ud {\bf t}=\Gamma_t^n \sum_{\rho \in S_n} \int_{t_{\rho(1)}<\ldots<t_{\rho(n)}}\psi_{t,x,z}^{(n)}({\bf t},{\bf t})\ud {\bf t}.
\end{equation}
If $t_{\rho(1)}<\ldots<t_{\rho(n)}<t=t_{\rho(n+1)}$ then by \eqref{E:Fourier-gt},
\begin{eqnarray*}
\lefteqn{|\cF (g_{{\bf t},t,x+z}^{(n)}-g_{{\bf t},t,x}^{(n)})(\xi_1, \ldots,\xi_n)|^2=\lambda^{2n} \prod_{k=1}^{n} \exp \left( - \frac{t_{\rho(k+1)}-t_{\rho(k)}}{t_{\rho(k)}t_{\rho(k+1)}}
\left|\sum_{j=1}^{k}t_{\rho(j)}\xi_{\rho(j)}\right|^2 \right) } \\
& & \times\left|\exp\left[-\frac{i}{t}\left(\sum_{k=1}^{n}t_k \xi_k\right)\cdot (x+z) \right]
\int_{\bR^d} \exp\left\{-i \left[\sum_{k=1}^{n}\left(1-\frac{t_k}{t}\right)\xi_k\right] \cdot x_0\right\} G(t,x+z-x_0) u_0(\ud x_0)\right.\\
& & \left. -\exp\left[-\frac{i}{t}\left(\sum_{k=1}^{n}t_k \xi_k\right)\cdot x \right]
\int_{\bR^d} \exp\left\{-i \left[\sum_{k=1}^{n}\left(1-\frac{t_k}{t}\right)\xi_k\right] \cdot x_0\right\} G(t,x-x_0) u_0(\ud x_0) \right|^2.
\end{eqnarray*}
Inside the squared modulus above, we add and subtract the term
$$\exp\left[-\frac{i}{t}\left(\sum_{k=1}^{n}t_k \xi_k\right)\cdot x \right]\int_{\bR^d} \exp\left\{-i \left[\sum_{k=1}^{n}\left(1-\frac{t_k}{t}\right)\xi_k\right] \cdot x_0\right\} G(t,x+z-x_0) u_0(\ud x_0).$$
We obtain that
\begin{multline*}
|\cF (g_{{\bf t},t,x+z}^{(n)}-g_{{\bf t},t,x}^{(n)})(\xi_1, \ldots,\xi_n)|^2 \leq 2 \lambda^{2n} \prod_{k=1}^{n} \exp \left(- \frac{t_{\rho(k+1)}-t_{\rho(k)}}{t_{\rho(k)}t_{\rho(k+1)}}
\left|\sum_{j=1}^{k}t_{\rho(j)}\xi_{\rho(j)}\right|^2 \right)  \\
\times \left\{ \left|\exp\left[-\frac{i}{t}\left(\sum_{k=1}^{n}t_k \xi_k\right)\cdot z \right]-1\right|^2 J_0^2(t,x+z)+F^2(t,x,z)\right\},
\end{multline*}
where
\begin{equation}
\label{E:def-F}
F(t,x,z)=L(t,t,x,x+z)=\int_{\bR^d}|G(t,x+z-x_0)-G(t,x-x_0)|\:|u_0|(\ud x_0).
\end{equation}

Hence,
\begin{align*}
\psi_{t,x,z}^{(n)}({\bf t},{\bf t}) \leq & 2\lambda^{2n} \frac{1}{(2\pi)^{nd}}
\int_{\bR^{nd}}\mu(\ud \xi_1)\ldots \mu(\ud \xi_n) \prod_{k=1}^{n} \exp \left(- \frac{t_{\rho(k+1)}-t_{\rho(k)}}{t_{\rho(k)}t_{\rho(k+1)}}
\left|\sum_{j=1}^{k}t_{\rho(j)}\xi_{\rho(j)}\right|^2 \right) \\
 &\times  \left\{ \left|\exp\left[-\frac{i}{t}\left(\sum_{k=1}^{n}t_k \xi_k\right)\cdot z \right]-1\right|^2 J_0^2(t,x+z)+F^2(t,x,z)\right\}.
\end{align*}
Using \eqref{E:bound-Cn}, it follows that
\begin{equation}
\label{E:C2}
C_n(t,x,z)\leq 2 \left(C_n^{(1)}(t,x,z)+C_n^{(2)}(t,x,z) \right),
\end{equation}
where
\begin{align}
\nonumber
C_n^{(1)}(t,x,z)=& \Gamma_t^n \lambda^{2n} J_+^2(t,x+z) n! \frac{1}{(2\pi)^{nd}}\int_{0<t_1<\ldots<t_n<t} \int_{\bR^{nd}} \prod_{k=1}^{n}
\exp \left(- \frac{t_{k+1}-t_{k}}{t_{k}t_{k+1}}\left|\sum_{j=1}^{k}t_{j}\xi_{j}\right|^2 \right)
\\
\label{def-C1}
 &\times \left|\exp\left[-\frac{i}{t}\left(\sum_{k=1}^{n}t_k \xi_k\right)\cdot z \right]-1\right|^2 \mu(\ud \xi_1)\ldots \mu(\ud \xi_n) \ud t_1 \ldots \ud t_n,
\end{align}
and
\begin{align}
\nonumber
C_n^{(2)}(t,x,z)=& \Gamma_t^n \lambda^{2n} F^2(t,x,z) n! \frac{1}{(2\pi)^{nd}}\int_{0<t_1<\ldots<t_n<t} \int_{\bR^{nd}}
\exp \left(- \frac{t_{k+1}-t_{k}}{t_{k}t_{k+1}}\left|\sum_{j=1}^{k}t_{j}\xi_{j}\right|^2 \right) \\
\label{def-C2}
 &\times \mu(\ud \xi_1)\ldots \mu(\ud \xi_n) \ud t_1 \ldots \ud t_n.
\end{align}
By relation \eqref{E:L-conv}, $\lim_{|z|\to 0}J_+(t,x+z)=J_+(t,x)$. By the dominated convergence theorem,
$\lim_{|z|\to 0}C_{n}^{(1)}(t,x,z)=0$.
By \eqref{E:L-conv}, $\lim_{|z| \to 0}F(t,x,z)=0$, and hence $\lim_{|z|\to 0}C_{n}^{(2)}(t,x,z)=0$.
Relation \eqref{E:x-cont} follows from \eqref{E:C} and \eqref{E:C2}.
\end{proof}

\end{document}